\documentclass[a4paper]{amsart}
\usepackage[leqno]{amsmath}
\usepackage{enumerate}
\usepackage{hyphenat}
\usepackage{ifthen}
\usepackage{amstext,amsbsy,amsopn,amsthm,amsgen}
\usepackage{amsfonts,amscd,amsxtra,upref}
\usepackage{mathrsfs}
\usepackage{euscript}
\usepackage{amssymb}
\usepackage{hyperref}

\swapnumbers
\theoremstyle{plain}
\newtheorem{Thm}{Theorem}[section]
\newtheorem{Lem}[Thm]{Lemma}
\newtheorem{Cor}[Thm]{Corollary}
\newtheorem{Pro}[Thm]{Proposition}
\theoremstyle{definition}
\newtheorem{Def}[Thm]{Definition}
\theoremstyle{remark}
\newtheorem{Rem}[Thm]{Remark}

\numberwithin{equation}{section}

\newcommand{\ITE}[3]{\ifthenelse{#1}{#2}{#3}}\newcommand{\ITEE}[4][]{\ITE{\equal{#2}{#3}}{#4}{#1}}
 \ITE{\isundefined{\texorpdfstring}}{\newcommand{\texorpdfstring}[2]{#1}}{}

\newenvironment{cor}[2][]{\ITEE[{\begin{Cor}[#1]}]{#1}{}{\begin{Cor}}\label{cor:#2}}{\end{Cor}}
\newenvironment{dfn}[2][]{\ITEE[{\begin{Def}[#1]}]{#1}{}{\begin{Def}}\label{def:#2}}{\end{Def}}
\newenvironment{lem}[2][]{\ITEE[{\begin{Lem}[#1]}]{#1}{}{\begin{Lem}}\label{lem:#2}}{\end{Lem}}
\newenvironment{pro}[2][]{\ITEE[{\begin{Pro}[#1]}]{#1}{}{\begin{Pro}}\label{pro:#2}}{\end{Pro}}
\newenvironment{rem}[2][]{\ITEE[{\begin{Rem}[#1]}]{#1}{}{\begin{Rem}}\label{rem:#2}}{\end{Rem}}
\newenvironment{thm}[2][]{\ITEE[{\begin{Thm}[#1]}]{#1}{}{\begin{Thm}}\label{thm:#2}}{\end{Thm}}

\newcommand{\COR}[2][!]{\ITEE{#1}{!}{Corollary~}\ITEE{#1}{s}{Corollaries~}\textup{\ref{cor:#2}}}
\newcommand{\DEF}[2][!]{\ITEE{#1}{!}{Definition~}\ITEE{#1}{s}{Definitions~}\textup{\ref{def:#2}}}
\newcommand{\LEM}[2][!]{\ITEE{#1}{!}{Lemma~}\ITEE{#1}{s}{Lemmas~}\textup{\ref{lem:#2}}}
\newcommand{\PRO}[2][!]{\ITEE{#1}{!}{Proposition~}\ITEE{#1}{s}{Propositions~}\textup{\ref{pro:#2}}}
\newcommand{\THM}[2][!]{\ITEE{#1}{!}{Theorem~}\ITEE{#1}{s}{Theorems~}\textup{\ref{thm:#2}}}

\newcommand{\CCC}{\mathbb{C}}
\newcommand{\QQQ}{\mathbb{Q}}
\newcommand{\RRR}{\mathbb{R}}
\newcommand{\AAa}{\CMcal{A}}
\newcommand{\BBb}{\CMcal{B}}
\newcommand{\EEe}{\CMcal{E}}
\newcommand{\HHh}{\CMcal{H}}
\newcommand{\KKk}{\CMcal{K}}
\newcommand{\QQq}{\CMcal{Q}}
\newcommand{\ZZz}{\CMcal{Z}}
\newcommand{\DdD}{\EuScript{D}}
\newcommand{\EeE}{\EuScript{E}}
\newcommand{\MmM}{\EuScript{M}}
\newcommand{\WwW}{\EuScript{W}}
\newcommand{\Bb}{\mathfrak{B}}
\newcommand{\Mm}{\mathfrak{M}}
\newcommand{\bB}{\mathfrak{b}}
\newcommand{\jJ}{\mathfrak{j}}
\newcommand{\uU}{\mathfrak{u}}
\newcommand{\aaA}{\mathscr{A}}
\newcommand{\bbB}{\mathscr{B}}
\newcommand{\ddD}{\mathscr{D}}
\newcommand{\eeE}{\mathscr{E}}
\newcommand{\ffF}{\mathscr{F}}
\newcommand{\ggG}{\mathscr{G}}
\newcommand{\iiI}{\mathscr{I}}
\newcommand{\kkK}{\mathscr{K}}
\newcommand{\llL}{\mathscr{L}}
\newcommand{\mmM}{\mathscr{M}}
\newcommand{\rrR}{\mathscr{R}}
\newcommand{\ssS}{\mathscr{S}}
\newcommand{\uuU}{\mathscr{U}}
\newcommand{\vvV}{\mathscr{V}}
\newcommand{\zzZ}{\mathscr{Z}}
\newcommand{\eEE}{\pmb{E}}
\newcommand{\xXX}{\pmb{X}}
\newcommand{\yYY}{\pmb{Y}}

\newcommand{\ueA}{\textup{\textsf{A}}}
\newcommand{\ueB}{\textup{\textsf{B}}}
\newcommand{\ueS}{\textup{\textsf{S}}}
\newcommand{\ueT}{\textup{\textsf{T}}}
\newcommand{\ueW}{\textup{\textsf{W}}}
\newcommand{\ueX}{\textup{\textsf{X}}}
\newcommand{\ueY}{\textup{\textsf{Y}}}
\newcommand{\ueZ}{\textup{\textsf{Z}}}

\newcommand{\cdotaff}{\pmb{\cdot}}
\newcommand{\club}{{\scriptsize$\clubsuit$}}
\newcommand{\dd}{\colon}
\newcommand{\df}{\stackrel{\textup{def}}{=}}
\newcommand{\dint}[1]{\,\textup{d} #1}
\newcommand{\epsi}{\varepsilon}
\newcommand{\FTI}{\textup{\textsf{FTI}}}
\newcommand{\geqsl}{\geqslant}
\newcommand{\leqsl}{\leqslant}
\newcommand{\minusaff}{\pmb{-}}
\newcommand{\muple}{\stackrel{\textup{\tiny$\bullet$}}{}}
\newcommand{\plusaff}{\pmb{+}}
\newcommand{\scalar}[2]{\left\langle#1,#2\right\rangle}
\newcommand{\scalarr}{\langle\cdot,\mathrm{-}\rangle}
\newcommand{\spade}{{\scriptsize$\spadesuit$}}
\newcommand{\tI}{\textup{I}}
\newcommand{\varempty}{\varnothing}

\newcommand{\id}{\operatorname{id}}
\newcommand{\RE}{\operatorname{Re}}
\newcommand{\supp}{\operatorname{supp}}

\newcommand{\TFCAE}{The following conditions are equivalent:}
\newcommand{\iaoi}{if and only if}

\begin{document}

\title[Functional calculus in finite type I algebras]
 {Functional calculus in finite\\type I von Neumann algebras}
\author[P. Niemiec]{Piotr Niemiec}
\address{Instytut Matematyki\\{}Wydzia\l{} Matematyki i~Informatyki\\{}
 Uniwersytet Jagiello\'{n}ski\\{}ul. \L{}ojasiewicza 6\\{}30-348 Krak\'{o}w\\{}Poland}
\email{piotr.niemiec@uj.edu.pl}
\thanks{The author gratefully acknowledges the assistance of the Polish Ministry of Sciences
 and Higher Education grant NN201~546438 for the years 2010--2013.}
\begin{abstract}
A certain class of matrix-valued Borel matrix functions is introduced and it is shown that all
functions of that class naturally operate on any operator $T$ in a finite type~$\tI$ von Neumann
algebra $\MmM$ in a way such that uniformly bounded sequences $f_1,f_2,\ldots$ of functions that
converge pointwise to $0$ transform into sequences $f_1[T],f_2[T],\ldots$ of operators in $\MmM$
that converge to $0$ in the $*$-strong operator topology. It is also demonstrated that the double
$*$-commutant of any such operator $T$ which acts on a separable Hilbert space coincides with
the set of all operators of the form $f[T]$ where $f$ runs over all function from the aforementioned
class. Some conclusions concerning so-called operator-spectra of such operators are drawn and
a new variation of the spectral theorem for them is formulated.
\end{abstract}
\maketitle

\section{Introduction}

A classical (continuous or Borel) functional calculus for normal operators is a powerful tool
in operator theory (let us mention only a single deep application, namely, the functional calculus
for selfadjoint operators is involved in the most common proof of Kaplansky's density theorem for
von Neumann algebras; see \cite{kap}, Theorem~5.3.5 in \cite{kr1}, Theorem~1.9.1 in \cite{sak}
or Theorem~II.4.8 in \cite{ta1}). On the one hand, this calculus has intuitive properties (such as
$(g \circ f)[N] = g[f[N]]$ where $N$ is a normal operator and $f, g\dd \CCC \to \CCC$ are two Borel
functions) and is fairly flexible, and on the other hand, is, in a sense, as wide as possible and
sufficient enough. Indeed, if $N$ is a bounded normal operator on a separable Hilbert space, then
the set of all operators of the form $f[N]$ where $f$ runs over all bounded Borel functions $f\dd
\CCC \to \CCC$ coincides with the double commutant of $N$ (thanks to the Fuglede theorem and
Theorem~8.10 in Chapter~IX of \cite{con}). The aim of this paper is to generalize the concept
of functional calculus to a much wider class of operators (and their tuples)---those which belong to
(or are affiliated with) finite type~$\tI$ von Neumann algebras. Typical examples of such operators
are (finite square) block matrices of commuting normal operators, and direct sums of such block
matrices. To formulate our main result, we need to introduce necessary notions and fix
the notation.\par
For any $n > 0$, let $\mmM_n$ and $\uuU_n$ denote, respectively, the $C^*$-algebra of all complex
$n \times n$ matrices and its unitary group. Further, let $\mmM$ stand for the topological disjoint
union $\bigsqcup_{n=1}^{\infty} \mmM_n$ of the spaces $\mmM_n$. $\mmM$ is a Polish (that is,
separable completely metrizable) space. For any $X \in \mmM_n$ we put $d(X) \df n$. Further, for
each $U \in \uuU_n$, the matrix $U . X$ is defined as $U X U^{-1}$. Finally, if, in addition, $Y \in
\mmM_k$, $X \oplus Y$ denotes a block matrix $\begin{pmatrix}X & 0\\0 & Y\end{pmatrix} \in
\mmM_{n+k}$. Additionally, $\|X\|$ stands for the operator norm of $X$.\par
A function $f\dd \mmM \to \mmM$ is said to be \textit{compatible} if
\begin{itemize}
\item $d(f(X)) = d(X)$ for each $X \in \mmM$;
\item $f(U . X) = U . f(X)$ for any $X \in \mmM$ and $U \in \uuU_{d(X)}$;
\item $f(X \oplus Y) = f(X) \oplus f(Y)$ for all $X, Y \in \mmM$.
\end{itemize}
The function $f$ is \textit{locally bounded} if for any positive real constant $R$,
\begin{equation*}
\sup\{\|f(X)\|\dd\ X \in \bbB(R)\} < \infty
\end{equation*}
where $\bbB(R) \df \{X \in \mmM\dd\ \|X\| \leqsl R\}$. Polynomials in two noncommuting variables $X$
and $X^*$ are classical examples of locally bounded compatible functions.\par
Finally, let $\ffF_{1,1}^{loc}$ denote the family of all locally bounded compatible Borel functions
$f\dd \mmM \to \mmM$. It is easy to see that $\ffF_{1,1}^{loc}$ is a unital $*$-algebra when all
algebraic operations are defined pointwise.\par
The main result of the paper for single operators reads as follows. (Below $\id$ denotes
the identity map on $\mmM$ and $\BBb(\HHh)$ is the $C^*$-algebra of all bounded linear operators
on a Hilbert space $\HHh$.)

\begin{thm}{main1}
Let $T$ be an operator in a finite type~$\tI$ von Neumann algebra $\MmM$ that acts on a Hilbert
space $\HHh$. There exists a unique unital $*$-homomorphism $\ffF_{1,1}^{loc} \ni f \mapsto f[T] \in
\BBb(\HHh)$ such that $\id[T] = T$ and
\begin{itemize}
\item[(bc)] whenever $f_n \in \ffF_{1,1}^{loc}$ are uniformly bounded on each of the sets $\bbB(R)$
 and converge pointwise to $f \in \ffF_{1,1}^{loc}$, then $f_n[T]$ converge to $f[T]$
 in the $*$-strong topology of $\BBb(\HHh)$.
\end{itemize}
Moreover, $f[T] \in \MmM$ and $(g \circ f)[T] = g[f[T]]$ for any $f, g \in \ffF_{1,1}^{loc}$.\par
If, in addition, $\HHh$ is separable, then the set $\{f[T]\dd\ f \in \ffF_{1,1}^{loc}\}$ coincides
with the smallest von Neumann algebra on $\HHh$ that contains $T$.
\end{thm}

In \THM{main2} we shall prove a counterpart of \THM{main1} for finite tuples of operators affiliated
with a (common) finite type~$\tI$ von Neumann algebra.\par
As mentioned above, \THM{main1} applies to operators that are direct sums of finite block matrices
of commuting operators (and, up to unitary equivalence, only to such operators). This result enables
defining matrix-valued spectra of operators generating finite type~$\tI$ von Neumann algebras
in a new and transparent way. This shall be done in Section~4. For other results related with
the concept of an operator-valued spectrum, the reader is referred to, e.g., \cite{ern},
\cite{ha1,ha2}, \cite{kkl} and \cite{lee}. It turns out that the (classical) functional calculus for
normal operators is a special case of the calculus introduced in \THM{main1}. So, our concept is
more general, and as flexible and wide as the former.\par
Compatible functions, defined above, resemble (but differ from) \textit{nc} functions studied
in \cite{kvv}. Instead of preservig unitary actions, nc functions have to respect (simultaneous)
similarities. So, each nc function is compatible, but not conversely. Although the concepts
of compatible and nc functions are similar, both the works (\cite{kvv} and the present) are
of idependent interest (and none of them inspired the other). The notion of a compatible function
almost coincides with the notion of a \textit{decomposable} function, introduced in \cite{bfh} and
studied in more detail by Hadwin \cite{ha3} (see also \S5 in \cite{ha4}) who called the action
of decomposable functions on operators a \textit{functional calculus}. The main difference between
our approach and Hadwin's (beside his restriction to separable spaces and ours to finite type~$\tI$
von Neumann algebras) is that decomposable functions are, by their very definition, defined
on the algebra of all bounded operators on a separable Hilbert space and the scope of their
operation (that is, the class of their operands) is not extended to more general mathematical
objects, which is in contrast to the classical functional calculus for normal operators where
the functions that take part in this calculus are scalar and scalar-valued, and the way they operate
on scalars is extended to the class of all normal operators (with respective spectra). This is
exactly the spirit of our approach. If we agree that matrices are `simple objects' (in comparison
to linear operators acting in infinite-dimensional spaces), the conclusion of \THM{main1} says that
the scope of operation of (matrix-valued matrix) compatible functions may be extended, in a very
reasonable and intuitive way, to the class of all operators that belong to (or generate) finite
type~$\tI$ von Neumann algebras. Even more, decomposable functions have almost purely theoretical
meaning, because they are hard to construct (apart from polynomials in two noncommuting variables
$X$ and $X^*$ and holomorphic functions), whereas compatible functions, being merely Borel, are easy
to do so (of course, they need to satisfy some additional axioms, but---as shown
in Section~2---there is a natural one-to-one correspondence, which preserves boundedness and
pointwise convergence of sequences of functions, between all compatible functions and all
$\mmM$-valued Borel functions defined on a certain Borel subset of $\mmM$ that preserve the degree
of matrices). As decomposable functions are defined as operator-valued operator functions, there is
nothing challenging in studying issues similar to condition (bc) (formulated in \THM{main1}) for
them. And it was (bc) that gave foundations for our investigations. (Actually, when we were
preparing the material of the present paper, we were not aware of Hadwin's results. We learnt about
them only at the final stage of the work.)\par
Operators that generate finite type~$\tI$ von Neumann algebras have been widely studied for a long
time and the literature concerning this subject is rich. Here we mention only a few papers
in the topic (in the chronological order): \cite{bro}, \cite{gon}, \cite{pea}, \cite{rad},
\cite{b-d}, \S3 and \S4 of Chapter~5 in \cite{ern}, \cite{qui}.\par
It is assumed that the reader is familiar with the basics on von Neumann algebras as well as
operators affiliated with them. In particular, the reader should know that, according to a result
due to Murray and von Neumann \cite{mvn}, all operators affiliated with a fixed finite type~$\tI$
von Neumann algebra form a $*$-algebra (with natural algebraic operations). For another proof, see
e.g.\ \cite{liu}. Since our work strongly depends on reduction theory of von Neumann algebras due
to von Neumann himself \cite{jvn}, we expect that the reader knows a part of this theory that
concerns finite type~$\tI$ algebras.\par
The paper is organized as follows. Section~2 discusses, in more detail and more generally,
compatible functions. We prove there a result (\THM{F}) which the uniqueness part of \THM{main1}
depends on. The third part is devoted to the proof of a generalization (and a strenthening; see
\THM{main2}) of \THM{main1} (and of \THM{main1} itself). Section~4 discusses a new concept
of operator-valued spectra of operators. We define there the spectral measure of a finite tuple
$\ueT$ of operators affiliated with a finite type~$\tI$ von Neumann algebra, formulate a spectral
theorem for $\ueT$ (\THM{spectral}) and define the \textit{principal} spectrum of $\ueT$ as well as
give its characterization (\PRO{spectrum}), which resembles one of classical properties
of the (usual) spectrum of a normal operator. It is also shown how such a tuple makes
(in an ambiguous way) the underlying Hilbert space a left module over the $C^*$-product of all
$\mmM_n$.

\subsection*{Notation and terminology} In this paper all Hilbert spaces are complex.
An \textit{operator} means a closed densely defined (unless otherwise stated) linear operator acting
in a Hilbert space. An operator $T$ in a Hilbert space $\HHh$ is \textit{affiliated} with
a von Neumann algebra $\MmM$ of operators on $\HHh$ if $U T U^{-1} = T$ for any unitary operator $U$
from the commutant $\MmM'$ of $\MmM$. The algebra of all operators affiliated with a finite
type~$\tI$ von Neumann algebra $\MmM$ shall be denoted by $\hat{\MmM}$, and ``$\plusaff$'' and
``$\cdotaff$'' will stand for, respectively, the addition and the multiplication of $\hat{\MmM}$.
When $\ueT = (T_s)_{s \in S}$ is a system of bounded operators on a Hilbert space $\HHh$, we use
$\WwW(\ueT)$ to denote the smallest von Neumann algebra on $\HHh$ that contains each
of the operators $T_s$. If $T_s$ are merely closed and densely defined in $\HHh$ (and still $\ueT =
(T_s)_{s \in S}$), by $\WwW'(\ueT)$ we denote the set of all bounded operators $X$ on $\HHh$ for
which $X T_s \subset T_s X$ and $X^* T_s \subset T_s X^*$ for all $s \in S$. For any ring $\rrR$,
$\ZZz(\rrR)$ stands for the center of $\rrR$. We use $\WwW''(\ueT)$ to denote the commutant
$(\WwW'(\ueT))'$ of $\WwW'(\ueT)$. Recall that both $\WwW'(\ueT)$ and $\WwW''(\ueT)$ are von Neumann
algebras on $\HHh$, and that $\WwW''(\ueT) = \WwW(\ueT)$ (by von Neumann's double commutant theorem)
provided $\ueT$ consists of bounded operators. Whenever $x$ is a normal element in a unital
$C^*$-algebra and $u\dd D \to \CCC$ (where $D$ is a set in $\CCC$) is a continuous function whose
domain $D$ contains the spectrum $\sigma(x)$ of $x$, by $u[x]$ we denote the effect of acting of $u$
on $x$ within the classical functional calculus. For any topological space $X$, $\Bb(X)$ is reserved
to denote the $\sigma$-algebra of all Borel sets in $X$, that is, $\Bb(X)$ is the smallest
$\sigma$-algebra of subsets of $X$ that contains all open sets. A function $f\dd X \to Y$ between
topological spaces $X$ and $Y$ is \textit{Borel} if $f^{-1}(B) \in \Bb(X)$ for any $B \in
\Bb(Y)$.\par
All notations and terminologies introduced earlier in this section are obligatory. Additionally,
we denote by $I_n$ the unit $n \times n$ matrix.

\section{Compatible matrix functions}

First we shall extend the concept of compatible functions introduced in the previous part. To this
end, we reserve $\ell$ and $\ell'$ to denote lengths of tuples of matrices (as well as
of operators). So, $\ell$ and $\ell'$ are arbitrary positive integers. Let $\mmM^{(\ell)}$ stand for
the topological disjoint union $\bigsqcup_{n=1}^{\infty} \mmM_n^{\ell}$ of the product spaces
$\mmM_n^{\ell}$. $\mmM^{(\ell)}$ is a Polish (that is, separable completely metrizable) space.
Recall that $\mmM^{(1)} = \mmM$. The space $\mmM^{(\ell)}$ shall be equipped with ingredients $d$
(the \textit{degree} map), ``$\oplus$'' (the \textit{addition}) and ``$.$'' (the \textit{unitary
action}) defined below (according to the terminology of \cite{pn6}, $(\mmM^{(\ell)},d,\oplus,.)$ is
a \textit{tower}). For any $\ueA = (A_k)_{k=1}^{\ell} \in \mmM^{(\ell)}$ we denote by $d(\ueA)$
a common degree of the matrices $A_k$ (so, $d(\ueA) \df n$ for $\ueA \in \mmM_n^{\ell}$). Further,
for each $U \in \uuU_{d(\ueA)}$, the tuple $U . \ueA$ is defined as $(U A_k U^{-1})_{k=1}^{\ell} \in
\mmM^{(\ell)}$ (observe that $d(U . \ueA) = d(\ueA)$). Finally, if, in addition, also $\ueB =
(B_k)_{k=1}^{\ell}$ is a member of $\mmM^{(\ell)}$, the tuple $\ueA \oplus \ueB$ is defined
coordinatewise, that is, $\ueA \oplus \ueB \df (A_k \oplus B_k)_{k=1}^{\ell}$. Additionally,
$\|\ueA\|$ will stand for the maximum of all $\|A_k\|$.\par
When $\ssS$ is an arbitrary subset of $\mmM^{(\ell)}$, a function $f\dd \ssS \to \mmM^{(\ell')}$ is
said to be \textit{compatible} if
\begin{itemize}
\item $d(f(\ueX)) = d(\ueX)$ for each $\ueX \in \ssS$;
\item $f(U . (\bigoplus_{j=1}^s \ueX_j)) = U . (\bigoplus_{j=1}^s f(\ueX_j))$ for all finite systems
 $\ueX_1,\ldots,\ueX_s \in \ssS$ and $U \in \uuU_N$ with $N = \sum_{j=1}^s d(\ueX_j)$ such that
 $\bigoplus_{j=1}^s \ueX_j \in \ssS$.
\end{itemize}
(Compatible functions, in a more general context, were introduced in \cite{pn6} to show that certain
algebras of such functions serve as models for \textit{all} subhomogeneous $C^*$-algebras.)
The function $f$ is \textit{bounded} if $(\|f\| \df\,) \sup_{\ueX\in\ssS} \|f(\ueX)\| < \infty$
(where $\sup(\varempty) \df 0$), and $f$ is \textit{locally bounded} if, for any positive real
constant $R$, the restriction of $f$ to $\ssS \cap \bbB^{(\ell)}(R)$ is bounded where
$\bbB^{(\ell)}(R) \df \{\ueX \in \mmM^{(\ell)}\dd\ \|\ueX\| \leqsl R\}$. Finally, let
$\ffF_{\ell,\ell'}$ (resp.\ $\ffF_{\ell,\ell'}^{loc}$; $\ffF_{\ell,\ell'}^{bd}$) denote the family
of all (resp.\ all locally bounded; all bounded) compatible Borel functions $f\dd \mmM^{(\ell)} \to
\mmM^{(\ell')}$. It is easy to see that $\ffF_{\ell,\ell'}$ is a unital $*$-algebra when
the algebraic operations are defined as follows: if $f, g \in \ffF_{\ell,\ell'}$, $\alpha \in \CCC$,
$\ueX \in \mmM^{(\ell)}$ and $f(\ueX) = (A_1,\ldots,A_{\ell'})$ and $g(\ueX) =
(B_1,\ldots,B_{\ell'})$, then
\begin{align*}
(f+g)(\ueX) &\df (A_1+B_1,\ldots,A_{\ell'}+B_{\ell'}),\\
(\alpha \cdot f)(\ueX) &\df (\alpha A_1,\ldots,\alpha A_{\ell'}),\\
(f \cdot g)(\ueX) &\df (A_1 B_1,\ldots,A_{\ell'} B_{\ell'}),\\
(f^*)(\ueX) &\df (A_1^*,\ldots,A_{\ell'}^*)
\end{align*}
(the unit of $\ffF_{\ell,\ell'}$ is a function that is constantly equal to $(I_n,\ldots,I_n) \in
\mmM_n^{\ell'}$ on $\mmM_n^{\ell}$). Moreover, $\ffF_{\ell,\ell'}$ is sequentially closed
in the pointwise convergence topology of $(\mmM^{(\ell')})^{\mmM^{(\ell)}}$; that is, if $f_1,f_2,
\ldots \in \ffF_{\ell,\ell'}$ converge pointwise to $f\dd \mmM^{(\ell)} \to \mmM^{(\ell')}$, then $f
\in \ffF_{\ell,\ell'}$ as well. This implies that $\ffF_{\ell,\ell'}^{bd}$ is a unital $C^*$-algebra
(with the norm $\|\cdot\|$ defined above). For simplicity, we put $\ffF \df \bigcup_{\ell}
\bigsqcup_{\ell'} \ffF_{\ell,\ell'}$ where $\ell$ and $\ell'$ run over all positive integers.
Whenever $f_1,\ldots,f_{\ell'}$ is a finite system of functions in $\ffF_{\ell,1}$,
by $(f_k)_{k=1}^{\ell'}$ we shall denote their diagonal function; that is, $(f_k)_{k=1}^{\ell'}\dd
\mmM^{(\ell)} \to \mmM^{(\ell')}$ and $(f_k)_{k=1}^{\ell'}(\ueX) \df
(f_1(\ueX),\ldots,f_{\ell'}(\ueX))$ for $\ueX \in \mmM^{(\ell)}$. Finally, for any positive integer
$j \leqsl \ell$, $\pi^{(\ell)}_j\dd \mmM^{(\ell)} \to \mmM$ will stand for the projection onto
the $j$th coordinate; that is, $\pi^{(\ell)}_j(X_1,\ldots,X_{\ell}) \df X_j$.\par
Below we collect three additional properties of the class $\ffF$, which are relevant in our further
investigations. Their proofs are straightforward and thus we skip them.
\begin{itemize}
\item If $f_1,\ldots,f_{\ell'}$ is a finite collection of functions that belong to $\ffF_{\ell,1}$,
 then $(f_k)_{k=1}^{\ell'} \in \ffF_{\ell,\ell'}$.
\item If $f$ and $g$ belong to, respecrtively, $\ffF_{\ell,\ell'}$ and $\ffF_{\ell',\ell''}$, then
 $g \circ f$ is well defined and belongs to $\ffF_{\ell,\ell''}$.
\item If $f \in \ffF_{\ell,1}$ is such that $f(\ueX)$ is an invertible matrix for any $\ueX \in
 \mmM^{(\ell)}$, then the function $(f)^{-1}\dd \mmM^{(\ell)} \to \mmM$ defined by $(f)^{-1}(\ueX)
 \df (f(\ueX))^{-1}$ belongs to $\ffF_{\ell,1}$ as well.
\end{itemize}
The main result of this section is the following

\begin{thm}{F}
Let $\ffF_0$ denote one of $\ffF_{\ell,1}$ or $\ffF_{\ell,1}^{loc}$. Assume $\eeE$ is a unital
$*$-subalgebra of $\ffF_0$ which satisfies each of the following conditions:
\begin{enumerate}[\upshape(E1)]\addtocounter{enumi}{-1}
\item $\pi^{(\ell)}_j \in \eeE$ for each positive integer $j \leqsl \ell$;
\item if all values of $f \in \eeE$ are invertible matrices and $(f)^{-1} \in \ffF_0$, then
 $(f)^{-1}$ belongs to $\eeE$;
\item whenever $f_n \in \eeE$ are uniformly bounded and converge pointwise to a function $f \in
 \ffF_0$, then $f \in \eeE$.
\end{enumerate}
Then $\eeE = \ffF_0$.
\end{thm}

Although the above result is intuitive, its proof is a little bit complicated and based on a certain
selector theorem which shall be formulated after introducing necessary notions.\par
A tuple $\ueX \in \mmM^{(\ell)}$ is \textit{reducible} if there exist $\ueA, \ueB \in \mmM^{(\ell)}$
and a unitary matrix $U \in \uuU_{d(\ueX)}$ for which $\ueX = U . (\ueA \oplus \ueB)$; otherwise
$\ueX$ is \textit{irreducible} (see \cite{pn6}; the latter is similar to the notion
of an irreducible representation of a $C^*$-algebra). Two tuples $\ueA$ and $\ueB$
in $\mmM^{(\ell)}$ are said to be \textit{unitarily equivalent}, in symbols $\ueA \equiv \ueB$,
if $\ueB = V . \ueA$ for some $V \in \uuU_{d(\ueA)}$ (in particular, $d(\ueA) = d(\ueB)$ provided
$\ueA \equiv \ueB$). A set $\kkK \subset \mmM^{(\ell)}$ is called by us a \textit{kernel}
of $\mmM^{(\ell)}$ if $\kkK$ is a selector for irreducible tuples (with respect to the unitary
action); that is, $\kkK$ is a kernel if it consists of irreducible tuples and for any irreducible
tuple $\ueX \in \mmM^{(\ell)}$ there is a \textbf{unique} tuple $\ueY \in \kkK$ such that $\ueX
\equiv \ueY$.\par
What we need is the next result. We recall that a subset of a topological space is
\textit{$\sigma$-compact} if it is a countable union of compact sets.

\begin{pro}{kernel}
For any $\ell$, $\mmM^{(\ell)}$ contains a $\sigma$-compact kernel.
\end{pro}

The proof presented below is in the spirit of the argument from \cite{rad}. The existence of kernels
that are $\ggG_{\delta}$-sets in $\mmM^{(\ell)}$ may readily be deduced from Corollary~1 in \S2
of Chapter~XIV in \cite{k-m} or from \cite{cas}.

\begin{proof}
Denote by $\QQq$ the collection of all polynomials in $2\ell$ noncommuting variables whose all
coefficients belong to $\QQQ+\textup{i}\QQQ$, and by $\iiI^{(\ell)}$ the set of all tuples $\ueX \in
\mmM^{(\ell)}$ that are irreducible. It is easy to show that $\iiI^{(\ell)}$ is an open set
in $\mmM^{(\ell)}$ and therefore
\begin{equation}\label{eqn:irr}
\iiI^{(\ell)} \textup{ is $\sigma$-compact.}
\end{equation}
Note that a tuple $(X_1,\ldots,X_{\ell}) \in \mmM_n^{\ell}\ (\subset \mmM^{(\ell)})$ is irreducible
iff the set
\begin{equation*}
\{p(X_1,\ldots,X_{\ell},X_1^*,\ldots,X_{\ell}^*)\dd\ p \in \QQq\}
\end{equation*}
is dense in $\mmM_n$. Further, let $\vvV_0$ be the set of all matrices $X = [x_{jk}] \in \mmM$ such
that:
\begin{enumerate}[\upshape({a}x1)]
\item $x_{jk} + \bar{x}_{kj} = 0$ for all distinct $j$ and $k$;
\item $\RE(x_{jj}) > \RE(x_{kk})$ whenever $j > k$;
\item $x_{1k}$ is a positive real number for $k > 1$.
\end{enumerate}
Let us note here the following property of $\vvV_0$:
\begin{itemize}
\item[(P0)] if $X \in \vvV_0$, $Y = [y_{jk}] \in \mmM$ and $U \in \uuU_{d(X)}$ are such that
 \begin{enumerate}[({a}x1')]\addtocounter{enumi}{-1}
 \item $Y = U . X$; and
 \item $y_{jk} + \bar{y}_{kj} = 0$ for all distinct $j$ and $k$; and
 \item $\RE(y_{jj}) \geqsl \RE(y_{kk})$ whenever $j \geqsl k$; and
 \item $y_{1k}$ is a nonnegative real number for $k > 1$,
 \end{enumerate}
 then $Y = X$ and $U$ is a scalar multiple of the unit matrix.
\end{itemize}
To show (P0), assume $X = [x_{jk}]$. Since $Y = U . X$, we conclude that $(n \df\,) d(X) = d(Y)$ and
$Y+Y^* = U . (X+X^*)$. Consequently, the spectra of $X+X^*$ and $Y+Y^*$ (that is, the sets of all
their eigenvalues) coincide. But (ax1) and (ax1') imply that both $X+X^*$ and $Y+Y^*$ are diagonal
matrices. So, (ax2) and (ax2') yield that $\RE(x_{jj}) = \RE(y_{jj})$ for each $j \in
\{1,\ldots,n\}$, which means that $X+X^* = Y+Y^*$. Hence $U$ commutes with a diagonal matrix all
of whose diagonal entries are different. We infer that $U$ is a diagonal matrix as well. So,
if $n > 1$ and $\lambda_1,\ldots,\lambda_n$ are the consecutive entries of the diagonal of $U$, then
$y_{1k} = \lambda_1 x_{1k} \bar{\lambda}_k$ for each $k > 1$. Since both the numbers $y_{1k}$ and
$x_{1k}$ are real and nonnegative and $x_{1k} \neq 0$ (see (ax3) and (ax3')), we see that $\lambda_k
= \lambda_1$, which means that $U$ is a scalar multiple of the identity matrix and, consequently,
$Y = X$. The proof of (P0) is complete.\par
Now put $\vvV \df \{U . X\dd\ X \in \vvV_0,\ U \in \uuU_{d(X)}\}$. We claim that $\vvV$ is open
in $\mmM$. To convince oneself of that, it suffices to check that $\vvV \cap \mmM_N$ is open
in $\mmM_N$ for each $N > 0$. Suppose, on the contrary, that $\vvV \cap \mmM_N$ is not open. This
means that there is $X \in \vvV \cap \mmM_N$ and a sequence of matrices $X_n \in \mmM_N \setminus
\vvV$ that converge to $X$. Since any selfadjoint matrix is unitarily equivalent to diagonal, we see
that there is $U_n \in \uuU_N$ such that $U_n . (X_n+X_n^*)$ is a diagonal matrix whose consecutive
diagonal entries are monotone increasing (that is, nondecreasing). Further, there is a diagonal
matrix $V_n \in \uuU_N$ for which all entries, apart from the first, of the first row of $(V_n U_n)
. X_n$ are real and nonnegative. Observe that then $(V_n U_n) . (X_n+X_n^*) = U_n . (X_n+X_n^*)$
(because $V_n$ and $U_n . (X_n+X_n^*)$ are diagonal). Passing to a subsequence, we may and do assume
that $W_n \df V_n U_n$ converge to $W \in \uuU_N$. Then the matrix $Y \df W . X$ coincides with
$\lim_{n\to\infty} W_n . X_n$ and therefore has properties (ax1')--(ax3'). But $X \in \vvV$ which
means that $X' \df U . X$ belongs to $\vvV_0$ for some $U \in \uuU_N$. So, $Y = (W U^{-1}) . X'$ and
(P0) shows that $Y = X'$. This means that $Y \in \vvV_0$, from which we infer that all but a finite
number of the matrices $W_n . X_n$ also belong to $\vvV_0$, which contradicts the assumption that
$X_n \notin \vvV$. So, $\vvV$ is indeed open.\par
Of course, $\vvV$ meets each of $\mmM_n$. Now arrange all members of $\QQq$ in a sequence $p_1,p_2,
\ldots$ and for any positive integer $n$ denote by $\iiI_n$ the set of all $\ueX =
(X_1,\ldots,X_{\ell}) \in \iiI^{(\ell)}$ such that $p_n(X_1,\ldots,X_{\ell},X_1^*,\ldots,X_{\ell}^*)
\in \vvV$. Since $\vvV$ is open, we see that each of $\iiI_n$ is open in $\iiI^{(\ell)}$. Moreover,
we conclude from the previous remarks that the sets $\iiI_n$ cover $\iiI^{(\ell)}$. We put $\ddD_n
\df \iiI_n \setminus \iiI_{n-1}$ where $\iiI_0 \df \varempty$. Observe that:
\begin{enumerate}[(D1)]
\item the sets $\ddD_n$ are pairwise disjoint and cover $\iiI^{(\ell)}$; and
\item each of $\ddD_n$ is $\sigma$-compact (by \eqref{eqn:irr}); and
\item if $\ueX \in \ddD_n$ and $U \in \uuU_{d(\ueX)}$, then $U . \ueX \in \ddD_n$ as well.
\end{enumerate}
Let $\ueX \in \ddD_n$. Since $p_n(\ueX) \in \vvV$, there is $U_{\ueX} \in \uuU_{d(\ueX)}$ for which
$U_{\ueX} . p_n(\ueX) \in \vvV_0$ (moreover, $U_{\ueX}$ is unique up to a scalar multiple, thanks
to (P0)). We now put
\begin{equation*}
\kkK \df \{U_{\ueX} . \ueX\dd\ \ueX \in \iiI^{(\ell)}\}.
\end{equation*}
We see that $\kkK$ consists of irreducible tuples and for any $\ueX \in \iiI^{(\ell)}$, $\kkK$
contains a tuple unitarily equivalent to $\ueX$ (by (D1)). Further, notice that if $U_{\ueX} . \ueX$
belongs to $\ddD_n$, then $\ueX \in \ddD_n$ as well (by (D3)) and hence $p_n(U_{\ueX} . \ueX) =
U_{\ueX} . p_n(\ueX) \in \vvV_0$. This shows that $p_n(\ueA) \in \vvV_0$ for any $\ueA \in \kkK \cap
\ddD_n$. Now assume $\ueA$ and $\ueB$ are two tuples in $\kkK$ that are unitarily equivalent. Then,
by (D1) and (D3), there is a unique $n$ such that $\ueA, \ueB \in \ddD_n$. Take $V \in
\uuU_{d(\ueA)}$ for which $\ueB = V . \ueA$. Then also $p_n(\ueB) = V . p_n(\ueA)$. But, as shown
above, both $p_n(\ueA)$ and $p_n(\ueB)$ belong to $\vvV_0$. Thus, we infer from (P0) that $V$ is
a scalar multiple of the unit matrix, and hence $\ueY = \ueX$. This shows that $\kkK$ is a kernel.
It remains to check that $\kkK \cap \ddD_m \cap \mmM_n^{\ell}$ is $\sigma$-compact for each $m$ and
$n$ (because then $\kkK$ itself be $\sigma$-compact). To this end, we consider the map
\begin{equation*}
\Phi\dd \uuU_n \times \mmM_n^{\ell} \ni (U;X_1,\ldots,X_{\ell}) \mapsto
U . p_m(X_1,\ldots,X_{\ell},X_1^*,\ldots,X_{\ell}^*) \in \mmM_n.
\end{equation*}
Notice that $\Phi$ is proper (i.e., the inverse image of a compact set in $\mmM_n$ under $\Phi$ is
compact). It is an easy observation that $\vvV_0$ is $\sigma$-compact (it is even locally compact)
and therefore $\llL \df \Phi^{-1}(\vvV_0)$ is also $\sigma$-compact. Consequently, the set $\{U .
\ueX\dd\ (U,\ueX) \in \llL\}$ is $\sigma$-compact as well. But, the last aforementioned set
coincides with $\kkK$ (because for any $\ueX \in \ddD_m$ the unitary matrix $U_{\ueX}$ is unique
up to a scalar multiple) and we are done.
\end{proof}

We shall also need the following lemma, which is a special case of a variation
of the Stone-Weierstrass theorem for $C^*$-algebras proved in \cite{pn2}.

\begin{lem}{SW}
Let $X$ be a compact Hausdorff space and let $\AAa$ be a unital $C^*$-algebra. Let $\EeE$ be
a $*$-subalgebra of $C(X,\AAa)$ such that for any two points $x$ and $y$ of $X$ there is $f \in
\EeE$ with $f(x) = 1$ and $f(y) = 0$. Then the closure of $\EeE$ in $C(X,\AAa)$ coincides with
the $*$-algebra $\Delta_2(\EeE)$ of all maps $u \in C(X,\AAa)$ such that for any $x, y \in X$ and
each $\epsi > 0$ there exists $v \in \EeE$ with $\|v(x) - u(x)\| < \epsi$ and $\|v(y) - u(y)\| <
\epsi$.
\end{lem}

The reader interested in other results in a similar spirit as above is referred to Theorem~1.4
in \cite{fel} or Corollary~11.5.3 in \cite{dix}. For other variations of the Stone-Weierstrass
thorem settled in $C^*$-algebras, consult, e.g., \cite{gli}, \S4.7 in \cite{sak}, \cite{lon}
or \cite{pop}.

\begin{pro}{F}
Let $\kkK$ be a $\sigma$-compact kernel of $\mmM^{(\ell)}$ and let $\ffF_0$ denote either
the $*$-algebra of all Borel functions $u\dd \kkK \to \mmM$ such that
\begin{equation}\label{eqn:compat}
d(u(\ueX)) = d(\ueX) \qquad (\ueX \in \kkK)
\end{equation}
or the $*$-algebra of all locally bounded such functions. If $\EeE$ is a unital $*$-subalgebra
of $\ffF_0$ such that $\pi^{(\ell)}_k\bigr|_{\kkK} \in \EeE$ for any $k \leqsl \ell$ and conditions
\textup{(E1)--(E2)} hold, then $\EeE = \ffF_0$.
\end{pro}
\begin{proof}
Below we shall involve the concept of \textit{$\bB$-transform}, introduced by us in \cite{pn1}.
Namely, for any matrix $X \in \mmM_n$, let $\bB(X) \df X (I_n + |X|)^{-1}$ (where $|X| \df
\sqrt{X^* X}$). The $\bB$-transform for tuples of matrices (or operators) is defined coordinatewise,
that is, $\bB(X_1,\ldots,X_{\ell}) \df (\bB(X_1),\ldots,\bB(X_{\ell}))$. What is relevant for us is
that:
\begin{enumerate}[($\bB$1)]
\item for any $\ueX \in \mmM^{(\ell)}$, $\|\bB(\ueX)\| \leqsl 1$;
\item if $\ueX \in \mmM^{(\ell)}$ is irreducible, then so is $\bB(\ueX)$;
\item two tuples $\ueX$ and $\ueY$ in $\mmM^{(\ell)}$ are unitarily equivalent iff so are
 $\bB(\ueX)$ and $\bB(\ueY)$.
\end{enumerate}
We shall use these properties below, where $\bB$ shall be considered as a function
of $\kkK$ into $\mmM^{(\ell)}$. We shall also use $j$ to denote the unit of $\EeE$ ($j(\ueX) =
I_{d(\ueX)}$ for any $\ueX \in \kkK$). More generally, for any Borel set $\aaA$ in $\kkK$, we denote
by $j_{\aaA}$ the function of $\kkK$ into $\mmM$ that coincides with $j$ on $\aaA$ and vanishes
at each point of $\kkK \setminus \aaA$.\par
We fix a compact subset $\llL$ of $\kkK$. Our aim is to show that
\begin{itemize}
\item[(L0)] $j_{\llL} u \in \EeE$ for any bounded Borel function $u\dd \kkK \to \mmM$ that
 satisfies \eqref{eqn:compat}.
\end{itemize}
To this end, first observe that
\begin{itemize}
\item[($\star$)] for any two distinct element $\ueX$ and $\ueY$ in $\kkK$, there is a polynomial $p$
 in $2\ell$ noncommuting variables with $p(\bB(\ueX),\bB(\ueX)^*) = 0$ and $p(\bB(\ueY),\bB(\ueY)^*)
 = I_{d(\ueY)}$
\end{itemize}
(because $\bB(\ueX)$ and $\bB(\ueY)$ are irreducible and unitarily inequivalent, see ($\bB$2) and
($\bB$3); cf.\ Proposition~4.2.5 in \cite{dix}). Since $\pi^{(\ell)}_k\bigr|_{\kkK}$ belongs
to $\EeE$ for any $k$, (E1)--(E2) imply that
\begin{equation}\label{eqn:sub-b}
\pi^{(\ell)}_k \circ \bB \in \EeE
\end{equation}
as well. Indeed, $j+\bigl(\pi^{(\ell)}_k\bigr)^* \pi^{(\ell)}_k$ belongs to $\EeE$ and has all its
values invertible. Thus, it follows from (E1) that the function
\begin{equation*}
v_k\dd \kkK \ni (X_1,\ldots,X_{\ell}) \mapsto (I_{d(X_k)} + X_k^* X_k)^{-1} \in \mmM
\end{equation*}
is a member of $\EeE$. Observe that $v_k$ is a nonnegative (selfadjoint) element
of the $C^*$-algebra $\EeE_{bd}$ of all bounded functions in $\EeE$ whose norms do not exceed $1$
($\EeE_{bd}$ is indeed a $C^*$-algebra thanks to (E2)). So, $v_k' \df u[v_k]$ belongs to $\EeE_{bd}$
where $u\dd [0,1] \to [0,\infty)$ is given by $u(t) = \frac{\sqrt{t}}{\sqrt{t}+\sqrt{1-t}}$. Direct
calculations show that $v_k'(X_1,\ldots,X_{\ell}) = (I_{d(X_k)} + |X_k|)^{-1}$. Consequently,
$\pi^{(\ell)}_k\bigr|_{\kkK} v_k' \in \EeE$, which is equivalent to \eqref{eqn:sub-b}. Having this,
we conclude that the function $\tilde{p}\dd \kkK \ni \ueX \mapsto p(\bB(\ueX),\bB(\ueX)^*) \in \mmM$
belongs to $\EeE_{bd}$ when $p$ is any polynomial in $2\ell$ noncommuting variables. Replacing
$\tilde{p}$ by $\frac12(\tilde{p}+\tilde{p}^*)$, we see that for any $\ueX \in \llL$ and $\ueY \in
\kkK \setminus \llL$ there is a selfadjoint \textbf{continuous} function $f_{\ueX,\ueY} \in
\EeE_{bd}$ such that $f_{\ueX,\ueY}(\ueX) = 0$ and $f_{\ueX,\ueY}(\ueY) = I_{d(\ueY)}$ (see
($\star$)). We may also assume that $0 \leqsl f_{\ueX,\ueY} \leqsl j$, and $f_{\ueX,\ueY}(\ueZ) =
0$ and $f_{\ueX,\ueY}(\ueW) = I_{d(\ueW)}$ for $\ueZ$ and $\ueW$ from, respectively, some
neighbourhoods $U_{\ueX,\ueY}$ of $\ueX$ and $V_{\ueX,\ueY}$ of $\ueY$ (because we may replace,
if needed, $f_{\ueX,\ueY}$ by $w[f_{\ueX,\ueY}]$ where $w(t) = \max(\min(3t-1,1),0)$). Now
it follows from the compactness of $\llL$ that, when $\ueY$ is fixed, there are a finite number
of points $\ueX_1,\ldots,\ueX_n \in \kkK$ such that $\llL \subset \bigcup_{k=1}^n U_{\ueX_k}$.
We put $g_{\ueY} \df j - (\prod_{k=1}^n f_{\ueX_k,\ueY})^* (\prod_{k=1}^n f_{\ueX_k,\ueY})$.
Observe that $g_{\ueY}$ is a selfadjoint member of $\EeE$, $0 \leqsl g_{\ueY} \leqsl j$,
$g_{\ueY}(\ueZ) = I_{d(\ueZ)}$ for all $\ueZ \in \kkK$ and $g_{\ueY}$ vanishes at each point
of the neighbourhood $W_{\ueY} \df \bigcap_{k=1}^n V_{\ueX_k,\ueY}$ of $\ueY$. Further, we infer
from the separability of $\kkK$ there there is a sequence $\ueY_1,\ueY_2,\ldots$ of elements
of $\kkK \setminus \llL$ for which $\bigcup_{k=1}^{\infty} W_{\ueY_k} = \kkK \setminus \llL$. Notice
that then the functions $h_n \df \prod_{k=1}^n g_{\ueY_k}$ belong to $\EeE$, are uniformly bounded
and converge pointwise to $j_{\llL}$. So, (E2) yields that
\begin{equation}\label{eqn:jL}
j_{\llL} \in \EeE.
\end{equation}
Now let $\EeE'$ be the collection of all continuous functions from $\llL$ into $\mmM$ which are
restrictions of some functions from $\EeE_{bd}$. We want to show that $\EeE'$ coincides with
the $C^*$-algebra $\DdD$ of all continuous functions $u\dd \llL \to \mmM$ for which
\eqref{eqn:compat} is fulfilled for any $\ueX \in \llL$. To this end, take $N > 0$ such that $\llL
\subset \bigcup_{n=1}^N \mmM_n$, put $R \df N!$ and $\AAa \df \mmM_R$. We define a $*$-homomorphism
$\Phi\dd \DdD \to C(\llL,\AAa)$ by
\begin{equation*}
(\Phi(f))(\ueX) \df \underbrace{f(\ueX) \oplus \ldots \oplus f(\ueX)}_{R/d(\ueX)}.
\end{equation*}
Since $\Phi$ is isometric, it is enough to check that $\Phi(\EeE') = \Phi(\DdD)$. The previous
reasoning (which starts from ($\star$)) shows that for any distinct elements $\ueX$ and $\ueY$
of $\llL$ there is a function $f \in \EeE_{bd}$ with $f(\ueX) = I_{d(\ueX)}$ and $f(\ueY) = 0$. Then
$\Phi(f\bigr|_{\llL})(\ueX) = I_R$ and $\Phi(f\bigr|_{\llL})(\ueY) = 0$. So, \LEM{SW} implies that
$\Phi(\EeE')$ coincides with $\Delta_2(\Phi(\EeE')) \subset C(\llL,\AAa)$ ($\EeE'$ is
a $C^*$-algebra as the image of a $C^*$-subalgebra of $\EeE_{bd}$ under a $*$-homomorphism). But,
since $\bB(\ueX)$ is irreducible, for each matrix $T \in \mmM_{d(\ueX)}$ we may find a polynomial
$p$ in $2\ell$ noncommuting variables such that $\tilde{p}(\ueX) = T$. All these remarks show that
\begin{equation*}
\{(u(\ueX),u(\ueY))\dd\ u \in \Phi(\EeE')\} = \{(v(\ueX),v(\ueY))\dd\ v \in \Phi(\DdD)\}
\end{equation*}
and consequently $\Delta_2(\Phi(\EeE')) = \Phi(\DdD)$. So, $\EeE' = \DdD$. Combining this connection
with \eqref{eqn:jL}, we obtain
\begin{itemize}
\item[($\star\star$)] if $u \in \ffF_0$ vanishes at each point off $\llL$ and its restriction
 to $\llL$ is continuous, then $u \in \EeE$.
\end{itemize}
Further, denote by $\Mm$ the family of all Borel sets $\aaA$ in $\llL$ for which $j_{\aaA} \in
\EeE$. Since $\EeE$ is an algebra that satisfies (E2) and contains $j_{\llL}$, we readily get that
$\Mm$ is a $\sigma$-algebra of subsets of $\llL$. So, to conclude that $\Mm = \Bb(\llL)$, it is
enough to show that each closed subset of $\llL$ belongs to $\Mm$. But this simply follows from
($\star\star$). Indeed, if $\llL_0$ is a closed set in $\llL$, there is a sequence of continuous
functions $v_n\dd \llL \to [0,1]$ that converge pointwise to the characteristic function
of $\llL_0$. Then each of the functions $q_n\dd \kkK \to \mmM$ given by $q_n(\ueX) = v_n(\ueX)
j(\ueX)$ for $\ueX \in \llL$ and $q_n(\ueX) = 0$ otherwise belongs to $\EeE$, thanks
to ($\star\star$), and hence $j_{\llL_0} \in \EeE$, by (E2). So, $\Mm = \Bb(\llL)$ or,
equivalently,
\begin{equation}\label{eqn:M}
j_{\aaA} \in \EeE \qquad (\aaA \in \Bb(\llL)).
\end{equation}
Finally, let $u\dd \kkK \to \mmM$ be a bounded Borel function that satisfies \eqref{eqn:compat}. For
fixed $\epsi > 0$, we may find sequences $\bbB_1,\bbB_2,\ldots$ of pairwise disjoint Borel subsets
of $\llL$ that cover $\llL$ and $T_1,T_2,\ldots$ of matrices such that for any $n$ and each $\ueX
\in \bbB_n$, $d(\ueX) = d(T_n)$ and $\|u(\ueX) - T_n\| \leqsl \epsi$ (this implies that
$\sup_{n\geqsl1} \|T_n\| < \infty$). For each $n$, the function $\xi_n\dd \kkK \to \mmM$ which is
constantly equal to $T_n$ on $\bbB_n$ and vanishes at each point off $\bbB_n$ belongs to $\EeE$,
thanks to ($\star\star$) and \eqref{eqn:M}. We define $u_{\epsi}\dd \kkK \to \mmM$ as the pointwise
limit of the series $\sum_{n=1}^{\infty} \xi_n$. Since the partial sums of the aforementioned series
are uniformly bounded, we infer from (E2) that $u_{\epsi} \in \EeE$. Finally, it follows from our
construction that $\|j_{\llL} u - u_{\epsi}\| \leqsl \epsi$ and, consequently, $j_{\llL} u \in
\EeE$ (again by (E2)). This finishes the proof of (L0).\par
Having (L0), we can now easily finish the proof of the proposition. First take an arbitrary bounded
Borel function $u\dd \kkK \to \mmM$ that satisfies \eqref{eqn:compat}. Since $\kkK$ is
$\sigma$-compact, we may express $\kkK$ as the union of an ascending sequence of compact subsets
$\llL_1,\llL_2,\ldots$ of $\kkK$. Then $j_{\llL_n} u \in \EeE$, by (L0) and consequently $u$,
as the pointwise limit of a uniformly bounded sequence of the functions $j_{\llL_n} u$, belongs
to $\EeE$ as well (again by (E2)).\par
Finally, we take an arbitrary Borel function $u \in \ffF_0$. Let $v\dd \kkK \to \mmM$ and $w\dd \kkK
\to \mmM$ be defined by $v(\ueX) \df \bB(u(\ueX))$ and $w(\ueX) \df (I_{d(\ueX)} + |u(\ueX)|)^{-1}$.
The previous paragraph proves that $v, w \in \EeE$. Now (E1) implies that also $(w)^{-1}$ belongs
to $\EeE$. Since $u = v \cdot (w)^{-1}$, we conclude that $u \in \EeE$ and we are done.
\end{proof}

\begin{cor}{ext}
Let $\kkK$ be a $\sigma$-compact kernel of $\mmM^{(\ell)}$. Every Borel function $u\dd \kkK \to
\mmM$ that satisfies \eqref{eqn:compat} admits a unique extension $\hat{u}$ to a function
in $\ffF_{\ell,1}$. Moreover, the assignment $u \mapsto \hat{u}$ establishes a one-to-one
correspondence between Borel functions $u\dd \kkK \to \mmM$ satisfying \eqref{eqn:compat} and
functions from $\ffF_{\ell,1}$. Moreover,
\begin{enumerate}[\upshape(ext1)]
\item $\hat{u} \in \ffF_{\ell,1}^{bd}$ \textup{(}resp.\ $\hat{u} \in \ffF_{\ell,1}^{loc}$\textup{)}
 iff $u$ is bounded \textup{(}resp.\ $u$ is locally bounded\textup{)}, and $\|\hat{u}\| = \|u\|$;
\item Borel functions $u_n\dd \kkK \to \mmM$ \textup{(}for which \eqref{eqn:compat} hold\textup{)}
 converge pointwise to $u\dd \kkK \to \mmM$ \iaoi{} the functions $\hat{u}_n$ converge so
 to $\hat{u}$.
\end{enumerate}
\end{cor}
\begin{proof}
First of all, note that for any $\ueX \in \mmM^{(\ell)}$ there are $U \in \uuU_{d(\ueX)}$ and
$\ueX_1,\ldots,\ueX_p \in \kkK$ (for some $p > 0$) such that
\begin{equation}\label{eqn:pd}
\ueX = U . (\ueX_1 \oplus \ldots \oplus \ueX_p).
\end{equation}
This property implies that two functions from $\ffF_{\ell,1}$ coincide provided their restrictions
to $\kkK$ do so. Now let $\EeE$ consists of all Borel functions $u\dd \kkK \to \mmM$ that satisfy
\eqref{eqn:compat} and extend to some (necessarily unique) function $\hat{u} \in \ffF_{\ell,1}$.
The decomposition \eqref{eqn:pd} (for any $\ueX$) shows (ext1) for any $u \in \EeE$ and enables
proving that $\EeE$ satisfies all assumptions of \PRO{F}. So, $\EeE$ constists of all Borel
functions $u\dd \kkK \to \mmM$ that satisfy \eqref{eqn:compat} and, consequently, (ext2) holds
(again by \eqref{eqn:pd}), which finishes the proof.
\end{proof}

\begin{proof}[Proof of \THM{F}]
By \PRO{kernel}, there is a $\sigma$-compact kernel $\kkK$ for $\mmM^{(\ell)}$. Put $\EeE' \df
\{u\bigr|_{\kkK}\dd\ u \in \EeE\}$ and similarly $\ffF_0' \df \{u\bigr|_{\kkK}\dd\ u \in \ffF_0\}$,
and observe that conditions (E1)--(E2) are fulfilled for $\EeE'$ and $\ffF_0'$ (by \COR{ext}). So,
\PRO{F} yields that $\EeE' = \ffF_0'$. Thus, we conclude from \COR{ext} that $\EeE = \ffF_0$.
\end{proof}

The next result is an immediate consequence of \COR{ext}. The proof is skipped.

\begin{cor}{center}
Let $\ffF_0$ denote one of $\ffF_{\ell,1}$, $\ffF_{\ell,1}^{loc}$ or $\ffF_{\ell,1}^{bd}$.
The center $\ZZz(\ffF_0)$ of $\ffF_0$ coincides with the set of all functions $u \in \ffF_0$ such
that $u(\ueX)$ is a scalar multiple of the unit matrix for any irreducible $\ell$-tuple $\ueX \in
\mmM^{(\ell)}$.
\end{cor}

In the sequel we shall also need the next two results.

\begin{lem}{ext-kernel}
Every Borel set $\kkK_0$ in $\mmM_m^{(\ell)}$ that consists of mutually unitarily inequivalent
irreducible $\ell$-tuples is contained in a Borel kernel $\kkK$ of $\mmM^{(\ell)}$.
\end{lem}
\begin{proof}
Let $\zzZ$ be the center of $\uuU_m$. It follows e.g.\ from Theorem~1.2.4 in \cite{b-k} that there
is a Borel set $\ddD \subset \uuU_m$ which meets each coset of $\zzZ$ in exactly one point. This
implies that the function $\phi\dd \ddD \times \kkK_0 \ni (U,\ueX) \mapsto U . \ueX \in
\mmM_m^{(\ell)}$ is one-to-one (because members of $\kkK_0$ are irreducible). So, it follows from
a theorem of Suslin (see Corollary~A.7 in \cite{ta1} or Theorem~9 in \S1 of Chapter~XIII
in \cite{k-m}) that the image $\bbB$ of $\phi$ is a Borel set in $\mmM_m^{(\ell)}$. Hence, the set
$\llL \df \mmM^{(\ell)} \setminus \bbB$ is Borel in $\mmM^{(\ell)}$. Notice that $\llL \cap \kkK_0 =
\varempty$ and $\llL$ is unitarily invariant, that is, $U . \ueX \in \llL$ for any $\ueX \in \llL$
and $U \in \uuU_{d(\ueX)}$. Now if $\kkK_1$ is any Borel kernel for $\mmM^{(\ell)}$ (see
\PRO{kernel}), it suffices to put $\kkK \df \kkK_0 \cup (\llL \cap \kkK_1)$ to get the kernel
we searched for.
\end{proof}

The following result may be proved using methods and concepts of \cite{pn6}. Below we give
an alternative proof (especially that we shall use a part of it in the last section).

\begin{lem}{ext}
Every Borel function $u\dd \kkK \to \mmM$ defined on a Borel kernel $\kkK$ of $\mmM^{(\ell)}$
and satisfying \eqref{eqn:compat} admits a unique extension $\hat{u}$ to a function
in $\ffF_{\ell,1}$. Moreover, conditions \textup{(ext1)--(ext2)} hold.
\end{lem}
\begin{proof}
All we need to show is that $u$ admits an extension (see the proof of \COR{ext}). Arguing
as in the proof of \LEM{ext-kernel}, we see that for any $n > 0$ there exists a Borel set $\ddD_n
\subset \uuU_n$ which meets each coset of the center of $\uuU_n$ at exactly one point. We put $\ddD
\df \bigcup_{n=1}^{\infty} \ddD_n$. Since $\kkK$ consists of irreducible $\ell$-tuples, we see that
the function $\varphi\dd \bigcup_{n=1}^{\infty} (\ddD_n \times (\kkK \cap \mmM_n^{\ell})) \ni
(U,\ueX) \mapsto U . \ueX \in \mmM^{(\ell)}$ is one-to-one (it is also Borel). Now let $\kkK'$ be
a $\sigma$-compact kernel of $\mmM^{(\ell)}$. Then the set $\llL \df \varphi^{-1}(\kkK')$ is Borel
and the projection $\llL \ni (U,\ueX) \mapsto \ueX \in \kkK$ onto the second coordinate is
a bijection. Applying Suslin theorem (see the previous proof), we conclude that the function $g\dd
\kkK \to \ddD$ that assigns to each $\ueX \in \kkK$ the unique unitary matrix $U \in \ddD$ such that
$(U,\ueX) \in \llL$ is Borel. Notice that then
\begin{equation*}
\psi\dd \kkK \ni \ueX \mapsto g(\ueX) . \ueX \in \kkK'
\end{equation*}
is a well defined Borel isomorphism. Now it remains to apply \COR{ext} to the function $u'\dd \kkK'
\ni \ueX \mapsto (g \circ \psi^{-1})(\ueX) . (u \circ \psi^{-1})(\ueX) \in \mmM$ to obtain
the extension $\hat{u} \in \ffF_{\ell,1}$ of $u$ and then to check that $\hat{u}$ extends $u$ (which
is left to the reader).
\end{proof}

\section{Functional calculus}

This section is mainly devoted to the proof of the following

\begin{thm}{main2}
Let $\MmM$ be a finite type~$\tI$ von Neumann algebra in a Hilbert space $\HHh$. For any finite
tuple $\ueT = (T_1,\ldots,T_{\ell})$ of operators in $\hat{\MmM}$ there exists a unique function
that assigns to every function $u$ in $\ffF_{\ell,\ell'}$ an $\ell'$-tuple $u[T_1,\ldots,T_{\ell}]$
of operators that belong to $\hat{\MmM}$ in a way such that the following conditions hold:
\begin{enumerate}[\upshape(F1)]
\item $\pi^{(\ell)}_j[\ueT] = T_j$ for any $j \leqsl \ell$;
\item for any finite collection $\{f_1,\ldots,f_{\ell'}\}$ of functions in $\ffF_{\ell,1}$,
 \begin{equation*}
 (f_k)_{k=1}^{\ell'}[\ueT] = (f_1[\ueT],\ldots,f_{\ell'}[\ueT]);
 \end{equation*}
\item if $f \in \ffF_{\ell,1}^{bd}$, then $f[\ueT]$ is bounded;
\item the function $\ffF_{\ell,1} \ni u \mapsto u[\ueT] \in \hat{\MmM}$ is a $*$-homomorphism;
\item whenever $u_n \in \ffF_{\ell,1}^{bd}$ are uniformly bounded and converge pointwise to $u \in
 \ffF_{\ell,1}^{bd}$, then $u_n[\ueT]$ converge to $u[\ueT]$ in the $*$-strong operator topology
 of $\BBb(\HHh)$.
\end{enumerate}
Moreover,
\begin{enumerate}[\upshape(F1)]\addtocounter{enumi}{5}
\item $v[u[\ueT]] = (v \circ u)[\ueT]$ for any $u \in \ffF_{\ell,\ell'}$ and $v \in
 \ffF_{\ell',\ell''}$;
\item if $T_j \in \MmM$ for each $j$, then $u[\ueT]$ is bounded for any $u \in \ffF_{\ell,1}^{loc}$;
\item for any $u \in \ffF_{\ell,1}$, $u[\ueT] \in \hat{\WwW}$ where $\WwW = \WwW''(\ueT)$;
\item if $\HHh$ is separable, then $\{u[\ueT]\dd\ u \in \ffF_{\ell,1}^{bd}\} = \WwW$, $\{u[\ueT]\dd\
 u \in \ZZz(\ffF_{\ell,1}^{bd})\} = \ZZz(\WwW)$, $\{u[\ueT]\dd\ u \in \ffF_{\ell,1}\} = \hat{\WwW}$
 and $\{u[\ueT]\dd\ u \in \ZZz(\ffF_{\ell,1})\} = \ZZz(\hat{\WwW})$ \textup{(}where $\WwW =
 \WwW''(\ueT)$\textup{)}.
\end{enumerate}
\end{thm}

We precede the proof of \THM{main2} by a few auxiliary results.

\begin{lem}{aff}
Let $T_1,\ldots,T_{\ell}$ be operators in a Hilbert space $\HHh$. \TFCAE
\begin{enumerate}[\upshape(i)]
\item there is a finite type~$\tI$ von Neumann algebra $\MmM$ which each of $T_j$ is affiliated
 with;
\item $\WwW''(T_1,\ldots,T_{\ell})$ is finite and type~$\tI$.
\end{enumerate}
\end{lem}
\begin{proof}
We only need to check that (ii) is implied by (i). Observe that, when $\MmM$ is as specified in (i),
$\WwW''(T_1,\ldots,T_{\ell}) \subset \MmM$. So, the conclusion of (ii) follows from the fact that
a von Neumann subalgebra of a finite type~$\tI$ von Neumann algebra is also finite and type~$\tI$,
which may simply be deduced, e.g., from Proposition~III.1.5.14 in \cite{bla}.
\end{proof}

For simplicity, we introduce

\begin{dfn}{fin1}
An $\ell$-tuple $\ueT = (T_1,\ldots,T_{\ell})$ of operators in a Hilbert space $\HHh$ is said to be
\textit{\FTI} if $\WwW''(\ueT)$ is finite and type~$\tI$. For every unitary operator $U\dd \HHh \to
\KKk$ (between Hilbert spaces) we use $U . \ueT$ to denote the $\ell$-tuple $(U T_1 U^{-1},\ldots,
U T_{\ell} U^{-1})$ of operators in $\KKk$. A closed linear subspace $\EEe$ of $\HHh$ is
\textit{reducing} for $\ueT$ if $\EEe$ reduces each of $T_j$. If this happens, $\ueT\bigr|_{\EEe}$
is defined as the tuple $(T_1\bigr|_{\EEe},\ldots,T_{\ell}\bigr|_{\EEe})$ of operators in $\EEe$.
Direct sums of $\ell$-tuples of operators are defined coordinatewise; that is, if, for $s \in S$,
$\ueT^{(s)} = (T_1^{(s)},\ldots,T_{\ell}^{(s)})$, then $\bigoplus_{s \in S} \ueT^{(s)}$ denotes
the $\ell$-tuple $(\bigoplus_{s \in S} T_1^{(s)},\ldots,\bigoplus_{s \in S} T_{\ell}^{(s)})$.
Finally, if each of $T_j$ is bounded, we use $\|\ueT\|$ to denote the maximum of $\|T_j\|$;
otherwise we put $\|\ueT\| \df \infty$.
\end{dfn}

According to the terminology of \cite{pn1}, the next result asserts that \FTI{} $\ell$-tuples form
an \textit{ideal}. Its proof is given in \cite{pn3}.

\begin{lem}{ideal}
Let $\ueT$ be an \FTI{} $\ell$-tuple of operators in a Hilbert space $\HHh$.
\begin{enumerate}[\upshape(A)]
\item If $U\dd \HHh \to \KKk$ is unitary, then $U . \ueT$ is \FTI.
\item If $\EEe \subset \HHh$ is a \textup{(}nonzero\textup{)} reducing subspace for $\ueT$, then
 $\ueT\bigr|_{\EEe}$ is \FTI.
\item The direct sum of any collection of \FTI{} $\ell$-tuples is \FTI{} as well.
\end{enumerate}
\end{lem}

\begin{lem}{sep}
Let $\ueT$ be an $\ell$-tuple of bounded operators acting on a separable Hilbert space such that
$\WwW(\ueT)$ is type~$\tI_m$ \textup{(}with finite $m$\textup{)}, $\WwW'(\ueT)$ is commutative and
$\|\ueT\| = R$. Then there exist a unital $*$-homomorphism $\ffF_{\ell,1}^{loc} \ni u \mapsto
u[\ueT] \in \WwW(\ueT)$ and a probabilistic Borel measure $\lambda$ on $\mmM^{(\ell)}$ such that
\textup{(F1)} holds and
\begin{enumerate}[\upshape($\Lambda$1)]\addtocounter{enumi}{-1}
\item $\lambda$ is concentrated on the set $\ssS$ of all irreducible $\ell$-tuples in $\mmM_m^{\ell}
 \cap \bbB^{(\ell)}(R)$ and for any $u, v \in \ffF_{\ell,1}^{loc}$, $u[\ueT] = v[\ueT]$ iff $u$ and
 $v$ are equal $\lambda$-almost everywhere; and
\item whenever $u_n \in \ffF_{\ell,1}^{loc}$ are uniformly bounded on $\ssS$ and converge pointwise
 $\lambda$-almost everywhere to $u \in \ffF_{\ell,1}^{loc}$, then $u_n[\ueT]$ converge to $u[\ueT]$
 in the $*$-strong operator topology; and
\item if functions $u_n \in \ffF_{\ell,1}^{loc}$ are uniformly bounded on $\ssS$ and $u_n[\ueT]$
 converge pointwise to $0$, then there is a subsequence $(u_{\nu_n})_{n=1}^{\infty}$
 of $(u_n)_{n=1}^{\infty}$ such that the functions $u_{\nu_n}$ converge pointwise $\lambda$-almost
 everywhere to the zero function in $\ffF_{\ell,1}^{loc}$; and
\item for any $u \in \ffF_{\ell,1}^{loc}$, $\|u[\ueT]\| \leqsl \sup \{\|u(\ueX)\|\dd\ \ueX \in
 \ssS\}$; conversely, for any $S \in \WwW(\ueT)$ \textup{(}resp.\ $S \in \ZZz(\WwW(\ueT))$\textup{)}
 there exists $u \in \ffF_{\ell,1}^{bd}$ \textup{(}resp.\ $u \in \ZZz(\ffF_{\ell,1}^{bd})$\textup{)}
 with $u[\ueT] = S$ and $\|u\| \leqsl \|S\|$.
\end{enumerate}
\end{lem}
\begin{proof}
Since $\WwW(\ueT)$ is a type~$\tI_m$ von Neumann algebra acting on a separable Hilbert space (say
$\HHh$) and $\WwW'(\ueT)$ is commutative, it follows from reduction theory due to von Neumann
\cite{jvn} (see also Chapter~14 in \cite{kr2}, \S3.2 in \cite{sak}, \S8 of Chapter~IV in \cite{ta1})
that there exist a standard Borel space $(\Omega,\Mm)$, a probabilistic measure $\mu\dd \Mm \to
[0,1]$ and a unitary operator $U\dd \HHh \to L^2(\Omega,\mu,\CCC^m)$ such that the algebra $\{U S
U^{-1}\dd\ S \in \WwW(\ueT)\}$ coincides with the set of all bounded decomposable operators
on $L^2(\Omega,\mu,\CCC^m)$. Recall that $L^2(\Omega,\mu,\CCC^m)$ is the Hilbert space of all
measurable functions $f\dd \Omega \to \CCC^m$ for which $(\|f\|^2 =) \int_{\Omega} \|f(\omega)\|^2
\dint{\mu(\omega)} < \infty$, and a bounded operator $S$ on $L^2(\Omega,\mu,\CCC^m)$ is decomposable
if there is a bounded measurable function $u\dd \Omega \to \mmM_m$ such that $S = M_u$ where
$(M_u f)(\omega) \df u(\omega) f(\omega)$ for each $\omega \in \Omega$. The proof of Theorem~14.1.10
in \cite{kr2} shows that
\begin{itemize}
\item[(S)] if $w_n\dd \Omega \to \mmM_m$ are uniformly bounded measurable functions such that
 $M_{w_n}$ converge to $M_w$ in the strong operator topology, then there exists a subsequence
 $(w_{\nu_n})_{n=1}^{\infty}$ of $(w_n)_{n=1}^{\infty}$ such that the functions $w_{\nu_n}$ converge
 pointwise $\mu$-almost everywhere to $w$.
\end{itemize}
Now write $\ueT = (T_1,\ldots,T_{\ell})$ and denote by $\xi_j\dd \Omega \to \mmM_n$ a bounded
measurable function such that $U T_j U^{-1} = M_{\xi_j}$. We may and do assume that
\begin{equation}\label{eqn:normuj}
\|\xi_j\| \leqsl R
\end{equation}
for each $j$. We claim that there exists a set $Z \in \Mm$ such that $\mu(Z) = 0$ and
\begin{itemize}
\item[(\spade)] the function
 \begin{equation}\label{eqn:psi}
 \psi\dd \Omega \setminus Z \ni \omega \mapsto (\xi_1(\omega),\ldots,\xi_{\ell}(\omega)) \in
 \mmM_n^{\ell}
 \end{equation}
 is one-to-one and its range $\kkK_0$ is contained in $\ssS$ and consists of mutually unitarily
inequivalent irreducible $\ell$-tuples.
\end{itemize}
This property is well-known, but, for the reader's convenience, we give its proof, which is similar
to that of item~2 of Theorem~3.4 in \cite{ern}. Since $\Omega$ is standard, there is a sequence
$v_1,v_2,\ldots\dd \Omega \to \mmM_m$ of bounded measurable functions such that
\begin{itemize}
\item[(\club)] for any two distinct points $a$ and $b$ of $\Omega$, the set $\{(v_j(a),v_j(b))\dd\
 j > 0\}$ is dense in $\mmM_m \times \mmM_m$.
\end{itemize}
Further, we conclude from the property that $\WwW(M_{\xi_1},\ldots,M_{\xi_{\ell}})$ contains each
of $M_{v_n}$, Kaplansky's density theorem and property (S) that for any $j$ there are a sequence
$p^{(j)}_1,p^{(j)}_2,\ldots$ of polynomials in $2\ell$ noncommuting variables and a set $Z_j \in
\Mm$ such that $\mu(Z_j) = 0$ and
\begin{equation}\label{eqn:convg}
\lim_{n\to\infty}
p^{(j)}_n(\xi_1(\omega),\ldots,\xi_{\ell}(\omega),\xi_1^*(\omega),\ldots,\xi_{\ell}^*(\omega)) =
v_j(\omega)
\end{equation}
for any $\omega \in \Omega \setminus Z_j$. We put $Z \df \bigcup_{j=1}^{\infty} Z_j$. Notice that
$\mu(Z) = 0$ and \eqref{eqn:convg} holds for any $\omega \in \Omega \setminus Z$, which implies that
$\psi$, given by \eqref{eqn:psi}, is one-to-one (because the functions $v_n$ separate points
of $\Omega$, by (\club)). Moreover, if $\omega, \omega' \in \Omega \setminus Z$ and $W \in \uuU_m$
are such that $W . \psi(\omega) = \psi(\omega')$, then $W . v_j(\omega) = v_j(\omega')$ for any $j$
(again by \eqref{eqn:convg}) and hence $\omega = \omega'$ (thanks to (\club)). Further, it follows
from \eqref{eqn:normuj} that the range $\kkK_0$ of $\psi$ is contained in $\mmM_m^{\ell} \cap
\bbB^{(\ell)}(R)$ and therefore it remains to check that each value of $\psi$ is an irreducible
$\ell$-tuple. But this again follows from \eqref{eqn:convg} and (\club), because the former formula
implies that $\WwW(\psi(\omega))$ coincides with $\mmM_m$ for any $\omega \notin Z$. So, the proof
of (\spade) is complete.\par
Replacing $\Omega$ by $\Omega \setminus Z$, we may and do assume that $Z = \varempty$. Further,
since $\Omega$ and $\mmM_m^{\ell}$ are standard measure spaces and $\psi$ is a one-to-one measurable
function, we conclude that $\kkK_0$ is a Borel subset of $\mmM_m^{\ell}$ and $\psi$ is a Borel
isomorphism of $\Omega$ onto $\kkK_0$ (consult, for example, Corollary~A.7 in \cite{ta1}
or Theorem~9 in \S1 of Chapter~XIII in \cite{k-m}). So, it follows from \LEM{ext-kernel} that there
is a Borel kernel $\kkK$ of $\mmM^{(\ell)}$ which contains $\kkK_0$. We now define a probabilistic
measure $\lambda\dd \Bb(\mmM^{(\ell)}) \to [0,1]$ as the transport of $\mu$ under $\psi$; that is,
$\lambda(\bbB) \df \mu(\psi^{-1}(\bbB))$ for any Borel set $\bbB$ in $\mmM^{(\ell)}$. We see that
$\lambda(\kkK_0) = 1$ and therefore the first claim of ($\Lambda$0) holds (see (\spade)). Finally,
for each $u \in \ffF_{\ell,1}^{loc}$, we define $u[\ueT]$ by $u[\ueT] \df U^{-1} M_{u \circ \psi} U\
(\in \WwW(\ueT))$. It is readily seen that the assignment $u \mapsto u[\ueT]$ correctly defines
a unital $*$-homomorphism for which (F1), the second claim of ($\Lambda$1) and the first
of ($\Lambda$3) hold (recall that $\pi^{(\ell)}_j \circ \psi = \xi_j$). Furthermore, for any $S \in
\WwW(\ueT)$ (resp.\ $S \in \ZZz(\WwW(\ueT))$) there is a bounded Borel function $v\dd \Omega \to
\mmM_m$ (resp.\ $v\dd \Omega \to \CCC \cdot I_m$) such that $U S U^{-1} = M_v$ and $\|v\| \leqsl
\|S\|$. Then let $u_0\dd \kkK \to \mmM$ coincide with $v \circ \psi^{-1}$ on $\kkK_0$ and vanish
at each point of $\kkK \setminus \kkK_0$. We then infer from \LEM{ext} (and \COR{center}) that there
is $u \in \ffF_{\ell,1}^{bd}$ (resp.\ $u \in \ZZz(\ffF_{\ell,1}^{bd})$) which extends $u_0$ and
satisfies $\|u\| \leqsl \|S\|$. This implies that $u \circ \psi = v$ and hence $u[\ueT] = M_v$. So,
the whole assertion of ($\Lambda$3) holds and thus it remains to verify conditions ($\Lambda$1) and
($\Lambda$2).\par
If $u_n$ and $u$ are as specified in ($\Lambda$1), then the functions $u_n \circ \psi$ are uniformly
bounded and converge pointwise $\mu$-almost everywhere to $u \circ \psi$ (by the definition
of $\lambda$). Then for each $g \in L^2(\Omega,\mu,\CCC^m)$ we have $\|u_n(\psi(\omega))
g(\omega)\|^2 \leqsl C \|g(\omega)\|^2$ for each $n$ and almost all $\omega \in \Omega$ (where $C$
is a positive contant independent of $n$) and therefore, by Lebesgue's dominated convergence
theorem, $\lim_{n\to\infty} \int_{\Omega} \|u_n(\psi(\omega)) g(\omega) - u(\psi(\omega))
g(\omega)\|^2 \dint{\mu(\omega)} = 0$. This shows that the operators $M_{u_n \circ \psi}$ converge
to $M_{u \circ \psi}$ in the strong operator topology. Consequently, $u_n[\ueT]$ converge so
to $u[\ueT]$. But also the functions $u_n^*$ are uniformly bounded on $\ssS$ and converge pointwise
$\lambda$-almost everywhere to $u^*$. We thus conclude that $u_n^*[\ueT]$ converge pointwise
to $u[\ueT]$, which finishes the proof of ($\Lambda$1).\par
Finally, assume $u_n$ are as specified in ($\Lambda$2). Then the functions $u_n \circ \psi$ are
uniformly bounded and $M_{u_n \circ \psi}$ converge pointwise to $0$. Hence, we infer from (S) that
there are a subsequence $(u_{\nu_n})_{n=1}^{\infty}$ of $(u_n)_{n=1}^{\infty}$ and a set $B \in \Mm$
such that $\mu(B) = 1$ and $\lim_{n\to\infty} u_{\nu_n}(\psi(\omega)) = 0$ for all $\omega \in B$.
Then the set $\bbB \df \psi(B)$ is Borel in $\mmM^{(\ell)}$, $\lambda(\bbB) = 1$ and
$\lim_{n\to\infty} u_n(\ueX) = 0$ for any $\ueX \in \bbB$.
\end{proof}

Now we shall easily prove a generalization of \LEM{sep}.

\begin{pro}{sep}
Let $\ueT$ be an \FTI{} $\ell$-tuple of bounded operators acting on a separable Hilbert space with
$\|\ueT\| = R$. Then there exist a unital $*$-homomorphism $\ffF_{\ell,1}^{loc} \ni u \mapsto
u[\ueT] \in \WwW(\ueT)$ and a probabilistic Borel measure $\lambda$ on $\mmM^{(\ell)}$ such that
\begin{itemize}
\item[($\Lambda$0')] $\lambda$ is concentrated on the set $\ssS$ of all irreducible $\ell$-tuples
 in $\bbB^{(\ell)}(R)$, and for any $u, v \in \ffF_{\ell,1}^{loc}$, $u[\ueT] = v[\ueT]$ iff $u$ and
 $v$ are equal $\lambda$-almost everywhere
\end{itemize}
and conditions \textup{(F1)} as well as \textup{($\Lambda$1)--($\Lambda$3)} hold.
\end{pro}
\begin{proof}
Assume $\ueT$ acts on $\HHh$. There is a sequence $\HHh_1,\HHh_2,\ldots$ (finite or not) of reducing
subspaces for $\ueT$ such that $\HHh = \bigoplus_{n\geqsl1} \HHh_n$, $\WwW(\ueT^{(n)})$ is type
$\tI_{p_n}$ (for some $p_n > 0$) and $\WwW'(\ueT^{(n)})$ is commutative for any $n$ where
$\ueT^{(n)} \df \ueT\bigr|_{\HHh_n}$ (to convince oneself that such a decomposition exists, consult,
for example, Theorem~3.6.1 in \cite{pn1}). Now to each of $\ueT^{(n)}$ we apply \LEM{sep} to obtain
a respective probabilistic measure $\lambda_n$ and a $*$-homomorphism $u \mapsto u[\ueT^{(n)}]$
of $\ffF_{\ell,1}^{loc}$ onto $\WwW(\ueT^{(n)})$. For any $u \in \ffF_{\ell,1}^{loc}$ we now put
$u[\ueT] \df \bigoplus_{n\geqsl1} u[\ueT^{(n)}]$. We also put $\lambda \df \sum_{n=1}^{\infty}
2^{-n} \lambda_n$. We see that conditions ($\Lambda$0') and ($\Lambda$1) as well as the first claim
of ($\Lambda$3) hold (and thus $u[\ueT]$ is a bounded operator). It is also readily seen that
the assignment $u \mapsto u[\ueT]$ correctly defines a unital $*$-homomorphism
of $\ffF_{\ell,1}^{loc}$ into $\BBb(\HHh)$. To conclude that in fact $u[\ueT] \in \WwW(\ueT)$ for
any $u$, we employ \THM{F}: the family $\EeE \df \{u \in \ffF_{\ell,1}^{loc}\dd\ u[\ueT] \in
\WwW(\ueT)\}$ satisfies all assumptions of the aforementioned result for $\ffF_0 =
\ffF_{\ell,1}^{loc}$ (thanks to ($\Lambda$1)), and hence $\EeE = \ffF_{\ell,1}^{loc}$.\par
We turn to ($\Lambda$2). Assume $u_n$ are as specified there. Then $u_n[\ueT^{(m)}]$ converge to $0$
in the strong operator topology (when $n$ tends to $\infty$) for each fixed $m$. So, using
the diagonal argument, we conclude that there is a subsequence $(u_{\nu_n})_{n=1}^{\infty}$ such
that the functions $u_{\nu_n}$ converge pointwise $\lambda_k$-almost everywhere to the zero function
in $\ffF_{\ell,1}^{loc}$ for each $k$. Since the set $\bbB_0 \df \{\ueX \in \mmM^{(\ell)}\dd\
\lim_{n\to\infty} \|u_{\nu_n}(\ueX)\| = 0\}$ is Borel, we see from the definition of $\lambda$ that
$\lambda(\bbB_0) = 1$, which proves ($\Lambda$2).\par
Finally, we turn to the remainder of ($\Lambda$3). To this end, let $S \in \WwW(\ueT)$ (resp.\
$S \in \ZZz(\WwW(\ueT))$). It follows from Kaplansky's density theorem that there are polynomials
$p_n \in \ffF_{\ell,1}^{loc}$ in $2\ell$ noncommuting variables such that the operators $p_n[\ueT]$
converge to $S$ in the strong operator topology. Then (by ($\Lambda$3) for $\ueT^{(k)}$) for any $k$
there is a function $u_k \in \ffF_{\ell,1}^{bd}$ (resp.\ $u_k \in \ZZz(\ffF_{\ell,1}^{bd})$) such
that $\|u_k\| \leqsl \|S\|$ and
\begin{equation*}
\lim_{n\to\infty} (p_n[\ueT^{(k)}] - u_k[\ueT^{(k)}]) = 0.
\end{equation*}
We now employ ($\Lambda$2) for each $k$. Using again the diagonal argument and passing
to a subsequence, we may assume that the matrix functions $p_n$ converge pointwise
$\lambda_k$-almost everywhere to $u_k$ for any $k$. Let $\ddD$ consist of all $\ell$-tuples $\ueX
\in \mmM^{(\ell)}$ for which the sequence $(p_n(\ueX))_{n=1}^{\infty}$ is convergent to a matrix
(resp.\ to a scalar multiple of the unit matrix) whose norm does not exceed $\|S\|$. Then $\ddD \in
\Bb(\mmM^{(\ell)})$ and $\lambda_k(\ddD) = 1$. Consequently, $\lambda(\ddD) = 1$ as well. Let $u'\dd
\ddD \to \mmM$ be the pointwise limit of the polynomials $p_n$. We claim that there is $u \in
\ffF_{\ell,1}^{bd}$ (resp.\ $u \in \ZZz(\ffF_{\ell,1}^{bd})$) that extends $u'$ and has the same
norm as $u'$. Assume we have such a function $u$. Then $u$ and $u_k$ are equal $\lambda_k$-almost
everywhere and thus $u[\ueT^{(k)}] = u_k[\ueT^{(k)}]$ (by ($\Lambda$0)). So, $u[\ueT^{(k)}]$ is
the limit of $p_n[\ueT^{(k)}]$ in the strong operator topology. Consequently, $u[\ueT] =
\bigoplus_{k=1}^{\infty} u[\ueT^{(k)}]$ is the limit of $p_n[\ueT]$ and hence $S = u[\ueT]$, which
completes the proof of ($\Lambda$3). So, we see it is enough to show the existence of $u$.\par
First note that $\|u'\| \leqsl \|S\|$, by the very definition of $\ddD$. Take a $\sigma$-compact
kernel $\kkK$ of $\mmM^{(\ell)}$ and observe that:
\begin{enumerate}[($\ddD$1)]
\item if $\ueX, \ueY \in \mmM^{(\ell)}$ are such that $\ueX \oplus \ueY \in \ddD$, then both $\ueX$
 and $\ueY$ belong to $\ddD$; and
\item whenever $\ueX_1,\ldots,\ueX_s$ belong to $\ddD$ and $V$ belongs to $\uuU_N$ with $N =
 \sum_{j=1}^s d(\ueX_j)$, then $V . (\bigoplus_{j=1}^s \ueX_j) \in \ddD$ and $u'(V .
 (\bigoplus_{j=1}^s \ueX_j)) = V . (\bigoplus_{j=1}^s u'(\ueX_j))$.
\end{enumerate}
These two properties imply that for any $\ueX \in \ddD$ there are $V \in \uuU_{d(\ueX)}$ and some
$\ueX_1,\ldots,\ueX_p \in \kkK \cap \ddD$ with $\ueX = V . (\bigoplus_{j=1}^p \ueX_j)$. Now let
$u_0\dd \kkK \to \mmM$ be a function that coincides with $u'$ on $\kkK \cap \ddD$ and vanishes
at each point of $\kkK \setminus \ddD$. It follows from \COR{ext} (and \COR{center}) that there is
$u \in \ffF_{\ell,1}^{bd}$ (resp.\ $u \in \ZZz(\ffF_{\ell,1}^{bd})$) which extends $u_0$ and has
the same norm as $u_0$. Then automatically $u$ extends $u'$ (by ($\ddD$2)) and $\|u\| = \|u'\|$.
\end{proof}

\begin{rem}{Linfty}
It follows from \PRO{sep} that for any \FTI{} $\ell$-tuple $\ueT$ of bounded operators
on a separable Hilbert space, there exists a probabilistic Borel measure $\lambda$
on $\mmM^{(\ell)}$ such that the assignment $u \mapsto u[\ueT]$ defines a $*$-isomorphism
of $L^{\infty}_{cm}(\lambda)$ onto $\WwW(\ueT)$ where $L^{\infty}_{cm}(\lambda)$ is the quotient
$C^*$-algebra $\ffF_{\ell,1}^{bd} / \lambda$ consisting of all equivalence classes of (all) bounded
compatible Borel functions from $\mmM^{(\ell)}$ into $\mmM$ with respect to $\lambda$-almost
everywhere equality.
\end{rem}

\begin{thm}{nonsep}
Let $\ueT$ be an \FTI{} $\ell$-tuple of bounded operators on a Hilbert space $\HHh$. There exists
a unique unital $*$-homomorphism $\ffF_{\ell,1}^{loc} \ni u \mapsto u[\ueT] \in \BBb(\HHh)$ such
that \textup{(F1)} holds and
\begin{itemize}
\item[($\Lambda$1')] whenever $u_n \in \ffF_{\ell,1}^{loc}$ are uniformly bounded
 on $\bbB^{(\ell)}(\|\ueT\|)$ and converge pointwise to $u \in \ffF_{\ell,1}^{loc}$, then
 $u_n[\ueT]$ converge to $u[\ueT]$ in the $*$-strong operator topology.
\end{itemize}
Moreover,
\begin{itemize}
\item[($\Lambda$3')] $\|u[\ueT]\| \leqsl \{\|u(\ueX)\|\dd\ \ueX \in \bbB^{(\ell)}(\|\ueT\|)\}$ and
 $u[\ueT] \in \WwW(\ueT)$ for any $u \in \ffF_{\ell,1}^{loc}$; and $v[\ueT] \in \ZZz(\WwW(\ueT))$
 for each $v \in \ZZz(\ffF_{\ell,1}^{loc})$.
\end{itemize}
\end{thm}
\begin{proof}
It may readily be shown (and follows from Theorem~2.2.4 in \cite{pn1}) that there is a collection
$\{\HHh_s\}_{s \in S}$ of separable reducing subspaces for $\ueT$ such that $\HHh =
\bigoplus_{s \in S} \HHh_s$. Then each of $\ueT^{(s)} \df \ueT\bigr|_{\HHh_s}$ is \FTI{} as well
(see \LEM{ideal}) and therefore we have, by \PRO{sep}, respective unital $*$-homomorphisms
$u \mapsto u[\ueT^{(s)}]$. For any $u \in \ffF_{\ell,1}^{loc}$ we put $u[\ueT] \df
\bigoplus_{s \in S} u[\ueT^{(s)}]$. Condition ($\Lambda$3) (for each of $\ueT^{(s)}$) implies that
$u[\ueT]$ is bounded and the estimation for the norm of $u[\ueT]$ specified in ($\Lambda$3') holds.
It is also immediate that the assignment $u \mapsto u[\ueT]$ correctly defines a unital
$*$-homomorphism for which (F1) and ($\Lambda$1') hold. To show its uniqueness and that $u[\ueT] \in
\WwW(\ueT)$ for any $u \in \ffF_{\ell,1}^{loc}$, we employ \THM{F} with $\ffF_0 =
\ffF_{\ell,1}^{loc}$. To this end, assume $\Phi\dd \ffF_{\ell,1}^{loc} \to \BBb(\HHh)$ is a unital
$*$-homomorphism for which respective conditions (F1) and (F5) hold and put $\EeE \df \{u \in
\ffF_{\ell,1}^{loc}\dd\ \Phi(u) = u[\ueT] \in \WwW(\ueT)\}$. It is readily seen that $\EeE$ is
a unital $*$-subalgebra of $\ffF_{\ell,1}^{loc}$ which satisfies conditions (E0)--(E2) (mainly
thanks to ($\Lambda$1') and (F5)) and therefore $\EeE = \ffF_{\ell,1}^{loc}$, by \THM{F}. Finally,
if $v \in \ZZz(\ffF_{\ell,1}^{loc})$, then $v$ commutes with each of $\pi^{(\ell)}_1,
(\pi^{(\ell)}_1)^*,\ldots,\pi^{(\ell)}_{\ell},(\pi^{(\ell)}_{\ell})^*$ and hence $v[\ueT]$ commutes
with each of $T_1,T_1^*,\ldots,T_{\ell},T_{\ell}^*$ (where $(T_1,\ldots,T_{\ell}) = \ueT$).
Consequently, $v[\ueT] \in \WwW(\ueT) \cap \WwW'(\ueT) = \ZZz(\ueT)$ and we are done.
\end{proof}

Now we want to extend the functional calculus built in \THM{nonsep} to all functions
in $\ffF_{\ell,1}$. We shall do this with the aid of the next two results. The first of them is
very simple and was established in \cite{pn4}. To simplify its statement, let us recall a suggestive
notation introduced there. Whenever $\MmM$ is a fnite type~$\tI$ von Neumann algebra acting
on $\HHh$, $\{Z_j\}_{j \in J}$ is a countable collection of mutually orthogonal projections
in $\ZZz(\MmM)$ that sum up to the unit of $\MmM$ (in the strong operator topology) and
$\{S_j\}_{j \in J}$ is a collection of operators in $\MmM$, a (possibly unbounded) operator $T \df
\sum_{j \in J} S_j Z_j$ is defined as follows. The domain $\DdD$ of $T$ consists of all vectors
$x \in \HHh$ such that $\sum_{j \in J} \|S_j Z_j x\|^2 < \infty$, and for each $x \in \DdD$ we put
\begin{equation}\label{eqn:series}
T x \df \sum_{j \in J} S_j Z_j x.
\end{equation}
Note that the ranges of $S_j Z_j$ are mutually orthogonal and thus all summands of the series
appearing in \eqref{eqn:series} are mutually orthogonal as well.

\begin{lem}[\cite{pn4}]{series}
Let $\MmM$ be a finite type~$\tI$ von Neumann algebra and $\{Z_j\}_{j \in J}$ be a collection
of mutually orthogonal projections in $\ZZz(\MmM)$ that sum up to the unit of $\MmM$.
\begin{enumerate}[\upshape(A)]
\item For any collection $\{S_j\}_{j \in J}$ of operators in $\MmM$, the operator $\sum_{j \in J}
 S_j Z_j$ is closed, densely defined and affiliated with $\MmM$; and $(\sum_{j \in J} S_j Z_j)^* =
 \sum_{j \in J} S_j^* Z_j$.
\item For any two collections $\{S_j\}_{j \in J}$ and $\{T_j\}_{j \in J}$ of operators in $\MmM$,
 \begin{align*}
 \sum_{j \in J} (S_j+T_j) Z_j &= \Bigl(\sum_{j \in J} S_j Z_j\Bigr) \plusaff
 \Bigl(\sum_{j \in J} T_j Z_j\Bigr) \quad \textup{and}\\
 \sum_{j \in J} (S_j T_j) Z_j &= \Bigl(\sum_{j \in J} S_j Z_j\Bigr) \cdotaff
 \Bigl(\sum_{j \in J} T_j Z_j\Bigr)
 \end{align*}
\end{enumerate}
\textup{(}where the operations on the right-hand sides are those in $\hat{\MmM}$\textup{)}.
\end{lem}

For simplicity, let us call two functions $u$ and $v$ in $\ffF_{\ell,1}$ \textit{disjoint} if for
any irreducible $\ell$-tuple $\ueX \in \mmM^{(\ell)}$, at least one of $u(\ueX)$ and $v(\ueX)$ is
zero.

\begin{lem}{unbd}
Let $\ueT$ be an \FTI{} $\ell$-tuple of bounded operators on a Hilbert space $\HHh$.
\begin{enumerate}[\upshape(A)]
\item If $\{b_j\}_{j \in J}$ is a countable collection of mutually disjoint functions
 in $\ffF_{\ell,1}^{bd}$ and $\sum_{j \in J} b_j(\ueX) = I_{d(\ueX)}$ for each irreducible $\ueX \in
 \mmM^{(\ell)}$, then $\{b_j[\ueT]\}_{j \in J}$ is a collection of mutually orthogonal projections
 from $\ZZz(\WwW(\ueT))$ that sum up to the unit of $\WwW(\ueT)$.
\item Let $\{b_j\}_{j \in J}$ and $\{b_s'\}_{s \in S}$ be two countable collections of mutually
 disjoint functions in $\ffF_{\ell,1}^{bd}$ such that $\sum_{j \in J} b_j(\ueX) = \sum_{s \in S}
 b_s'(\ueX) = I_{d(\ueX)}$ for any irreducible $\ueX \in \mmM^{(\ell)}$. If $\{u_j\}_{j \in J}$ and
 $\{u_s'\}_{s \in S}$ are two collections of functions in $\ffF_{\ell,1}^{bd}$ such that
 \begin{equation}\label{eqn:u-u'}
 \sum_{j \in J} u_j(\ueX) b_j(\ueX) = \sum_{s \in S} u_s'(\ueX) b_s'(\ueX)
 \end{equation}
 for any irreducible $\ueX \in \mmM^{(\ell)}$, then $\sum_{j \in J} u_j[\ueT] b_j[\ueT]$ and
 $\sum_{s \in S} u_s'[\ueT] b_s'[\ueT]$ coincide and are affiliated with $\WwW(\ueT)$.
\end{enumerate}
\end{lem}
\begin{proof}
We start from (A). Observe that under the assumptions of (A), $b_j(\ueX) \in \{0,I_{d(\ueX)}\}$ for
any $j \in J$ and irreducible $\ueX \in \mmM^{(\ell)}$. So, it follows from \COR{center} that all
$b_j$ belong to the center of $\ffF_{\ell,1}^{loc}$. Consequently, $Z_j \df b_j[\ueT] \in
\ZZz(\WwW(\ueT))$, by \THM{nonsep}. Moreover, $b_j = b_j^* b_j$ (because both these functions
coincide on the set of all irreducible $\ell$-tuples), and $b_j b_{j'} = 0$ for distinct $j$ and
$j'$, which implies that $Z_j$ are mutually orthogonal projections. Finally, since the partial sums
of $\sum_{j \in J} b_j$ are uniformly bounded and converge pointwise to the unit $\jJ$
of $\ffF_{\ell,1}$, we conclude that the partial sums of $\sum_{j \in J} Z_j$ converge to the unit
of $\WwW(\ueT)$.\par
Now assume $b_j$, $b_s'$, $u_j$ and $u_s'$ are as specified in (B). We deduce from part (A) and
\LEM{series} that both the operators $R \df \sum_{j \in J} u_j[\ueT] b_j[\ueT]$ and $R' \df
\sum_{s \in S} u_s'[\ueT] b_s'[\ueT]$ are well defined and affiliated with $\WwW(\ueT)$. To show
that $R = R'$, for any $\lambda \df (j,s) \in \Lambda \df J \times S$ we put $v_{\lambda} \df b_j
b_s'$. Observe that $u_j v_{(j,s)} = u_s' v_{(j,s)}$ (by \eqref{eqn:u-u'}), $u_j b_j =
\sum_{s \in S} u_j v_{(j,s)}$ and similarly $u_s' b_s' = \sum_{j \in J} u_s' v_{(j,s)}$ (and partial
sums of both these series are uniformly bounded), from which
we deduce that:
\begin{itemize}
\item $u_j[\ueT] v_{(j,s)}[\ueT] = u_s'[\ueT] v_{(j,s)}[\ueT]$; and
\item $u_j[\ueT] b_j[\ueT]$ is the limit of $\sum_{s \in S} u_j[\ueT] v_{(j,s)}[\ueT]$
 in the $*$-strong operator topology; and
\item $u_s'[\ueT] b_s'[\ueT]$ is the limit of $\sum_{j \in J} u_s'[\ueT] v_{(j,s)}[\ueT]$
 in the $*$-strong operator topology.
\end{itemize}
Since the summands of each of the series that appear above have mutually orthogonal disjoint ranges,
we conclude that for any $x \in \HHh$,
\begin{equation*}
\sum_{j \in J} \|u_j[\ueT] b_j[\ueT] x\|^2 = \sum_{j \in J} \sum_{s \in S} \|u_j[\ueT]
v_{(j,s)}[\ueT] x\|^2 = \sum_{s \in S} \|u_s'[\ueT] b_s'[\ueT] x\|^2,
\end{equation*}
which shows that the domains of $R$ and $R'$ coincide. An analogous argument proves that for any $x$
in this common domain, $\sum_{j \in J} u_j[\ueT] b_j[\ueT] x = \sum_{s \in S} u_s'[\ueT]
b_s'[\ueT] x$ and we are done.
\end{proof}

\begin{pro}{affil}
Let $\ueT$ be an \FTI{} $\ell$-tuple of bounded operators on a Hilbert space $\HHh$. There exists
a unique unital $*$-homomorphism $\ffF_{\ell,1} \ni f \mapsto f[\ueT] \in \hat{\WwW}$ such that
\textup{(F1)}, \textup{(F3)} and \textup{(F5)} hold where $\WwW \df \WwW(\ueT)$. Moreover, if $\HHh$
is separable, then
\begin{gather*}
\{u[\ueT]\dd\ u \in \ffF_{\ell,1}\} = \hat{\WwW},\\
\{u[\ueT]\dd\ u \in \ZZz(\ffF_{\ell,1})\} = \ZZz(\hat{\WwW}).
\end{gather*}
\end{pro}
\begin{proof}
The uniqueness part, as usual, follows from \THM{F} and is left to the reader. Here we focus only
on the existence part. Let $\Phi\dd \ffF_{\ell,1}^{loc} \ni u \mapsto u[\ueT] \in \WwW$ be a unital
$*$-homomorphism guaranteed by \THM{nonsep}. Fix an arbitrary $u \in \ffF_{\ell,1}$. For any
$n > 0$, let $\bbB_n$ consist of all irreducible $\ell$-tuples $\ueX \in \mmM^{(\ell)}$ such that
$n-1 \leqsl \|u(\ueX)\| < n$. Observe that $U . \ueX \in \bbB_n$ for each $\ueX \in \bbB_n$ and
$U \in \uuU_{d(\ueX)}$. We therefore conclude (using \COR{ext}) that there exists $b_n \in
\ffF_{\ell,1}^{bd}$ which vanishes at each irreducible $\ell$-tuple from $\mmM^{(\ell)} \setminus
\bbB_n$ and satisfies $b_n(\ueX) = I_{d(\ueX)}$ for any $\ueX \in \bbB_n$. Then $u b_n \in
\ffF_{\ell,1}^{bd}$ for all $n$, $b_n$ are mutually disjoint and $u$ is the pointwise limit
of the series $\sum_{n=1}^{\infty} (u b_n) b_n$. We define $u[\ueT]$ as $\sum_{n=1}^{\infty}
\Phi(u b_n) \Phi(b_n)$. It follows from \LEM{unbd} that $u[\ueT] \in \hat{\WwW}$. A standard
argument proves that the aforementioned operator $u[\ueT]$ coincides with $\Phi(u)$ for $u \in
\ffF_{\ell,1}^{loc}$. Consequently, conditions (F1), (F3) and (F5) are fulfilled. To check that
the assignment $u \mapsto u[\ueT]$ defines a $*$-homomorphism from $\ffF_{\ell,1}$ into
$\hat{\WwW}$, fix $u$ and $v$ in $\ffF_{\ell,1}$. We argue similarly as before. Let $\bbB_n'$
constist of all irreducible $\ueX \in \mmM^{(\ell)}$ for which $n-1 \leqsl
\max(\|u(\ueX)\|,\|v(\ueX)\|) < n$ and $b_n' \in \ffF_{\ell,1}^{bd}$ be a function that vanishes
at each irreducible $\ell$-tuple from $\mmM^{(\ell)} \setminus \bbB_n'$ and satisfies $b_n'(\ueX) =
I_{d(\ueX)}$ for any $\ueX \in \bbB_n'$. Then the functions $u b_n'$ and $v b_n'$ are bounded and
thus \LEM{unbd} implies that $u[\ueT] = \sum_{n=1}^{\infty} \Phi(u b_n') \Phi(b_n')$ and $v[\ueT] =
\sum_{n=1}^{\infty} \Phi(v b_n') \Phi(b_n')$. A similar reasoning shows that also $(uv)[\ueT] =
\sum_{n=1}^{\infty} \Phi(uv b_n') \Phi(b_n')$ and $(\alpha u + \beta v)[\ueT] = \sum_{n=1}^{\infty}
\Phi((\alpha u + \beta v) b_n') \Phi(b_n')$ for all scalars $\alpha, \beta \in \CCC$. Now it is
enough to apply \LEM{series}.\par
We turn to the additional claim of the proposition. Assume $\HHh$ is separable. Since for the centra
the proof goes similarly (because $\ZZz(\hat{\WwW}) = \hat{\ZZz}$ where $\ZZz = \ZZz(\WwW)$, see
\cite{pn4}), we shall show only the first additional conclusion. To this end, we fix $S \in
\hat{\WwW}$. It was shown in \cite{pn4} that then there exist $A \in \WwW$ and a sequence $Z_1,
Z_2,\ldots$ of mutually orthogonal projections in $\ZZz(\WwW)$ that sum up to the unit of $\WwW$ and
satisfy $S = \sum_{n=1}^{\infty} n A Z_n$. Further, let $\lambda$ be a probabilistic measure
as specified in \PRO{sep}. We conclude from that result that there are functions $v \in
\ffF_{\ell,1}^{bd}$ and $b_n \in \ZZz(\ffF_{\ell,1}^{bd})$ such that $v[\ueT] = A$, $b_n[\ueT] =
Z_n$ and $\|b_n\| \leqsl 1$. Further, since $b_n[\ueT] b_m[\ueT] = 0$ for distinct $n$ and $m$ and
$b_n^2[\ueT] = b_n[\ueT]$, property ($\Lambda$0') yields that $b_n b_m = 0$ and $b_n^2 = b_n$
$\lambda$-almost everywhere. Let $\aaA \in \Bb(\mmM^{(\ell)})$ be a set of full $\lambda$-measure
on which all the aforementioned equations hold. Then $b_n(\ueX) \in \{0,I_{d(\ueX)}\}$ for any
irreducible $X \in \aaA$. We leave it as an exercise that we may modify the sequence $b_1,
b_2,\ldots$ to obtain mutually disjoint functions $b_1',b_2',\ldots \in \ffF_{\ell,1}^{bd}$ such
that $b_n'$ and $b_n$ are equal $\lambda$-almost everywhere and $\sum_{n=1}^{\infty} b_n(\ueX) =
I_{d(\ueX)}$ for each irreducible $\ell$-tuple $\ueX \in \mmM^{(\ell)}$. Then $b_n'[\ueT] =
b_n[\ueT] = Z_n$, the pointwise limit $u$ of $\sum_{n=1}^{\infty} n v b_n'$ belongs
to $\ffF_{\ell,1}$ and \LEM{unbd} yields that $u[\ueT] = \sum_{n=1}^{\infty} n v[\ueT] b_n'[\ueT] =
S$, which finishes the proof.
\end{proof}

\begin{proof}[Proof of \THM{main2}]
As usual, uniqueness follows from \THM{F}. First we shall show existence and after that we shall
establish properties (F6)--(F9). To avoid misundestandings, the functional calculus $u \mapsto
u[\ueT]$ for $\ell$-tuples $\ueT$ of bounded operators obtained in \PRO{affil} shall be denoted
by $u \mapsto \Phi_{\ueT}(u)$. It follows from the results of \cite{pn4} that there is a sequence
$\HHh_1,\HHh_2,\ldots$ of mutually orthogonal reducing subspaces for $\ueT$ such that $\HHh =
\bigoplus_{n=1}^{\infty} \HHh_n$ and $\ueT^{(n)} \df \ueT\bigr|_{\HHh_n}$ consists of bounded
operators. For any $u \in \ffF_{\ell,1}$ we define $u[\ueT]$ as $\bigoplus_{n=1}^{\infty}
\Phi_{\ueT^{(n)}}(u)$. Since $\ffF_{\ell,1}^{bd}$ is a $C^*$-algebra (and the functions
$\Phi_{\ueT^{(n)}}$ are $*$-homomorphisms), we infer that $\|\Phi_{\ueT^{(n)}}(u)\| \leqsl \|u\|$
for any $u \in \ffF_{\ell,1}^{bd}$. This shows (F3), from which one deduces (F5); whereas (F1) and
(F4) are straightforward. Finally, defining, for any $u \in \ffF_{\ell,\ell'}$, $u[\ueT]$
as $((\pi^{(\ell')}_1 \circ u)[\ueT],\ldots,(\pi^{(\ell')}_{\ell'} \circ u)[\ueT])$, we see that
(F2) holds.\par
Further, (F7) is covered by \THM{nonsep}, whereas (F8) follows from \THM{F} (and (F5)). Also (F6)
follows from \THM{F}. Indeed, it suffices to check (F6) for $\ell'' = 1$. To this end, we fix $u \in
\ffF_{\ell,\ell'}$ and put $\EeE \df \{v \in \ffF_{\ell,1}\dd\ (v \circ u)[\ueT] = v[u[\ueT]]\}$.
We infer from (F2) and (F4)--(F5) that conditions (E0)--(E2) are satisfied. Thus,
$\EeE = \ffF_{\ell,1}$ and we are done.\par
We turn to (F9). Let $\uU \in \ffF_{\ell,\ell}$ be any function such that
\begin{equation*}
\uU(X_1,\ldots,X_{\ell}) =
(X_1 (I_{d(X_1)} - |X_1|)^{-1},\ldots,X_{\ell} (I_{d(X_{\ell})} - |X_{\ell}|)^{-1})
\end{equation*}
for any $(X_1,\ldots,X_{\ell}) \in \mmM^{(\ell)}$ with $\|X_j\| < 1$ for each $j$. Put $\ueS \df
\bB(\ueT)$. Then $\WwW = \WwW(\ueS)$. Below we shall think of the $\bB$-transform as of a function
in $\ffF_{\ell,\ell}$. Straightforward calculations show that $\uU(\bB(\ueX)) = \ueX$ for any $\ueX
\in \mmM^{(\ell)}$. It is easy to verify that $\bB[\ueT] = \ueS$ (use the uniqueness of the square
root of a nonnegative selfadjoint unbounded operator). Since $\ueS$ consists of bounded operators,
\PRO[s]{sep} and \PRO[]{affil} yield that
\begin{gather*}
\{f[\ueS]\dd\ f \in \ffF_{\ell,1}^{bd}\} = \WwW,\\
\{f[\ueS]\dd\ f \in \ZZz(\ffF_{\ell,1}^{bd})\} = \ZZz(\WwW),\\
\{f[\ueS]\dd\ f \in \ffF_{\ell,1}\} = \hat{\WwW},\\
\{f[\ueS]\dd\ f \in \ZZz(\ffF_{\ell,1})\} = \ZZz(\hat{\WwW}).
\end{gather*}
But (F6) implies that for any $f \in \ffF_{\ell,1}$, $f[\ueT] = (f \circ \uU)[\ueS]$ and thus (F9)
follows from the above formulas.
\end{proof}

\begin{proof}[Proof of \THM{main1}]
We leave it to the reader that all conclusions of the theorem follow from the results of this
section.
\end{proof}

Also the proof of the following result is skipped.

\begin{cor}{prop}
Let $f \in \ffF_{\ell,\ell'}$.
\begin{enumerate}[\upshape(A)]
\item If $\ueT$ is an \FTI{} $\ell$-tuple of operators in $\HHh$ and $U\dd \HHh \to \KKk$ is
 a unitary operator, then $f[U . \ueT] = U . f[\ueT]$.
\item If $\{\ueT^{(s)}\}_{s \in S}$ is an arbitrary collection of \FTI{} $\ell$-tuples, then
 $f[\bigoplus_{s \in S} \ueT^{(s)}] = \bigoplus_{s \in S} f[\ueT^{(s)}]$.
\item For any $\ueX \in \mmM^{(\ell)}$, $f[\ueX] = f(\ueX)$.
\end{enumerate}
\end{cor}

\section{Spectral theorem and spectrum}

In this section we apply the functional calculus built in the previous section to propose a new
approach to so-called operator-valued spectra and some variations of the spectral theorem.\par
We begin with

\begin{pro}{In}
Let $\ueT$ be an \FTI{} $\ell$-tuple of operators. For each $n > 0$, denote by $j_n$ the function
in $\ffF_{\ell,1}^{bd}$ which vanishes at each irreducible $\ell$-tuple in $\mmM^{(\ell)} \setminus
\mmM_n^{\ell}$ and sends each irreducible $\ell$-tuple in $\mmM_n^{\ell}$ to $I_n$. Then $j_n[\ueT]$
is the greatest selfadjoint projection $Z \in \ZZz(\WwW''(\ueT))$ such that $Z \WwW''(\ueT) =
\{Z A\dd\ A \in \WwW''(\ueT)\}$ is type~$\tI_n$.
\end{pro}
\begin{proof}
Since $\WwW \df \WwW''(\ueT)$ is finite and type~$\tI$, there is a sequence $Z_1,Z_2,\ldots$
of mutually orthogonal projections in $\ZZz(\WwW)$ that sum up to the unit of $\WwW$ and are such
that $Z_n \WwW$ is type~$\tI_n$ (or $Z_n = 0$). Our task is to show that $j_n[\ueT] = Z_n$. Denoting
by $\ueT^{(n)}$ the restriction of $\ueT$ to the range of $Z_n$, we obtain that $\ueT =
\bigoplus_{n=1}^{\infty} \ueT^{(n)}$ and $\WwW''(\ueT^{(n)})$ is type~$\tI_n$. Since $u[\ueT] =
\bigoplus_{n=1}^{\infty} u[\ueT^{(n)}]$ for any $u \in \ffF_{\ell,1}$ (by \COR{prop}), we see that
it suffices to check that $j_n[\ueT^{(n)}]$ coincides with the unit of $\WwW''(\ueT^{(n)})$ (that
is, with $Z_n$). This reduces the issue to the case when $\WwW$ is type $\tI_N$ (and we only need
to verify that $j_N[\ueT]$ is the unit of $\WwW$). As in the proof of \THM{main2}, we conclude from
the results of \cite{pn4} that $\ueT = \bigoplus_{s \in S} \ueT^{(s)}$ where each of $\ueT^{(s)}$
consists of bounded operators acting in a separable Hilbert space (and $\WwW''(\ueT^{(s)})$ is
type~$\tI_N$). Now condition ($\Lambda$0) of \LEM{sep} shows that $j_N[\ueT^{(s)}] =
\jJ[\ueT^{(s)}]$ for any $s \in S$ (recall that $\jJ$ is the unit of $\ffF_{\ell,1}$). Thus, again
thanks to \COR{prop}, $j_N[\ueT] = \jJ[\ueT]$ and we are done.
\end{proof}

\begin{lem}{zero}
Let $\ueT$ be an \FTI{} $\ell$-tuple and $f$ be any function in $\ffF_{\ell,\ell'}$. Let $u \in
\ffF_{\ell,1}^{bd}$ be a \textup{(}unique\textup{)} function such that for any irreducible
$\ell$-tuple $\ueX \in \mmM^{(\ell)}$, $u(\ueX)$ is the orthogonal projection onto the range
of $\sum_{j=1}^{\ell'} \bigl(\pi^{(\ell')}_j(f(\ueX))\bigr)^* \bigl(\pi^{(\ell')}_j(f(\ueX))\bigr)$.
Then $f[\ueT] = (0,\ldots,0)$ iff $u[\ueT] = 0$.
\end{lem}
\begin{proof}
Let $g \df \sum_{j=1}^{\ell'} \bigl(\pi^{(\ell')}_j \circ f\bigr)^* \cdot \bigl(\pi^{(\ell')}_j
\circ f\bigr)$. Since $0 \leqsl \bigl(\pi^{(\ell')}_j \circ f\bigr)^* \bigl(\pi^{(\ell')}_j \circ
f\bigr) \leqsl g$, it follows that $0 \leqsl \bigl(\pi^{(\ell')}_j[f[\ueT]]\bigr)^* \cdotaff
\bigl(\pi^{(\ell')}_j[f[\ueT]]\bigr) \leqsl g[\ueT]$ and, consequently, $f[\ueT] = 0$ \iaoi{}
$g[\ueT] = 0$. So, we only need to show that $g[\ueT] = 0$ iff $u[\ueT] = 0$. To this end, observe
that for any $\ueX \in \mmM^{(\ell)}$ there exists a positive integer $N > 0$ for which $2^{-N}
g(\ueX) \leqsl u(\ueX) \leqsl 2^N g(\ueX)$. One concludes that therefore there is a sequence $b_1,
b_2,\ldots \in \ffF_{\ell,1}^{bd}$ of mutually disjoint functions that sum up to the unit matrix
at each irreducible $\ell$-tuple and satisfy $2^{-n} g b_n \leqsl u b_n \leqsl 2^n g b_n$ (globally)
for any $n$. So, the assertion readily follows from the fact that $u[\ueT] = \sum_{n=1}^{\infty}
(u b_n)[\ueT] b_n[\ueT]$ and similarly $g[\ueT] = \sum_{n=1}^{\infty} (g b_n)[\ueT] b_n[\ueT]$.
\end{proof}

We would like to think of $u[\ueT]$ as of the integral $\int_{\mmM^{(\ell)}} u \dint{E}$ with
respect to some spectral measure $E$ (this shall be explained in more detail in the sequel). Under
such a thinking, it is a typical question of when $u[\ueT] = 0$ for a nonnegative function $u \in
\ffF_{\ell,1}$. The foregoing result reduces the above problem to functions that are selfadjoint
projections (as members of $C^*$-algebras).\par
\LEM{zero} is quite simple and intuitive. The following result is much more subtle.

\begin{pro}{zero}
For a function $f \in \ffF_{\ell,\ell'}$, let $v \in \ZZz(\ffF_{\ell,1}^{bd})$ be
a \textup{(}unique\textup{)} function such that for any irreducible $\ell$-tuple $\ueX \in
\mmM^{(\ell)}$, $v(\ueX) = 0$ if $f(\ueX) = (0,\ldots,0)$ and $v(\ueX) = I_{d(\ueX)}$ otherwise.
Then, for any \FTI{} $\ell$-tuple $\ueT$, $f[\ueT] = (0,\ldots,0)$ iff $v[\ueT] = 0$.
\end{pro}
\begin{proof}
Let $u \in \ffF_{\ell,1}^{bd}$ be as specified in \LEM{zero}, from which we infer that $f[\ueT] = 0$
iff $u[\ueT] = 0$. Observe that for any irreducible $\ueX \in \mmM^{(\ell)}$, $u(\ueT) = 0$
precisely when $v(\ueX) = 0$. One concludes that therefore
\begin{equation}\label{eqn:zero}
u(\ueX) = 0 \iff v(\ueX) = 0 \qquad (\ueX \in \mmM^{(\ell)}).
\end{equation}
Further, it follows from the results of \cite{pn4} that $\ueT = \bigoplus_{n=1}^{\infty} \ueT^{(n)}$
where each $\ueT^{(n)}$ consists of bounded operators (cf.\ the proof of \THM{main2}). So,
\COR{prop} implies that $u[\ueT] = 0$ iff $u[\ueT^{(n)}] = 0$ for all $n$ (and analogously for $v$
in place of $u$). This argument reduces the issue to the case when $\ueT$ is bounded. Then $\ueT =
\bigoplus_{s \in S} \ueT^{(s)}$ where each $\ueT^{(s)}$ acts on a separable Hilbert space.
So, arguing as before, we see that it suffices to prove the proposition for $\ell$-tuples $\ueT$
of bounded operators acting on a separable Hilbert space. In that case we apply \PRO{sep}. Let
a measure $\lambda$ be as specified there. It then follows from ($\Lambda$0') that $u[\ueT] = 0$
\iaoi{} $u = 0$ $\lambda$-almost everywhere (and the same for $v$ in place of $u$). So, a look
at \eqref{eqn:zero} finishes the proof.
\end{proof}

Noticing that selfadjoint projections in $\ZZz(\ffF_{\ell,1}^{bd})$ correspond to unitarily
invariant Borel sets of irreducible $\ell$-tuples of matrices, based on \PRO{zero}, we introduce

\begin{dfn}{meas}
Let $\Bb^{(\ell)}$ be the family of all Borel sets $\bbB$ of irreducible $\ell$-tuples of matrices
such that $U . \ueX \in \bbB$ whenever $\ueX \in \bbB$ and $U \in \uuU_{d(\ueX)}$ (every such a set
is called \textit{unitarily invariant}). $\Bb^{(\ell)}$ is a $\sigma$-algebra of subsets of the set
$\iiI^{(\ell)}$ of all irreducible $\ell$-tuples in $\mmM^{(\ell)}$. For any set $\bbB \in
\Bb^{(\ell)}$ there exists a unique function in $\ffF_{\ell,1}$, to be denoted by $\jJ_{\bbB}$, that
vanishes at each irreducible $\ell$-tuple in $\iiI^{(\ell)} \setminus \bbB$ and sends each member
of $\bbB$ to the unit matrix of a respective degree. For simplicity, we shall use $\iiI_n^{(\ell)}$
to denote the set $\iiI^{(\ell)} \cap \mmM_n^{\ell}$.\par
For any \FTI{} $\ell$-tuple $\ueT$ of operators, the \textit{spectral measure} of $\ueT$ is
the set function $E_{\ueT}\dd \Bb^{(\ell)} \to \ZZz(WwW(\ueT))$ given by $E_{\ueT}(\bbB) \df
\jJ_{\bbB}[\ueT]$ (that $E_{\ueT}(\bbB)$ belongs to $\ZZz(\WwW(\ueT))$ follows from property (F3)
in \THM{main2}).
\end{dfn}

We recall that an operator-valued set function $E\dd \Mm \to \BBb(\HHh)$ (where $\Mm$ is
a $\sigma$-algebra on a set $X$ and $\HHh$ is a Hilbert space) is said to be a \textit{spectral
measure} if
\begin{itemize}
\item $E(X)$ is the identity operator on $\HHh$; and
\item $E(\sigma)$ is an orthogonal projection for any $\sigma \in \Mm$; and
\item $E(\sigma \cap \sigma') = E(\sigma) E(\sigma')$ for all $\sigma, \sigma' \in \Mm$; and
\item whenever $\sigma_1,\sigma_2,\ldots$ are pairwise disjoint sets in $\Mm$, the series
 $\sum_{n=1}^{\infty} E(\sigma_n)$ converges to $E(\bigcup_{n=1}^{\infty} \sigma_n)$ in the strong
 operator topology.
\end{itemize}
The proofs of the next two result are left to the reader.

\begin{lem}{ET}
$\iiI^{(\ell)}$ is a locally compact Polish space and for any \FTI{} $\ell$-tuple $\ueT$, $E_{\ueT}$
is \textup{(}indeed\textup{)} a spectral measure. In particular, there exists the smallest set,
denoted by $\supp(E_{\ueT})$, among all sets $\bbB \in \Bb^{(\ell)}$ that are relatively closed
in $\iiI^{(\ell)}$ and satisfy $E_{\ueT}(\iiI^{(\ell)} \setminus \bbB) = 0$.
\end{lem}

\begin{cor}{equal}
Let $\ueT$ be an \FTI{} $\ell$-tuple. For any two functions $u$ and $v$ in $\ffF_{\ell,\ell'}$,
the set $D(u,v) \df \{\ueX \in \iiI^{(\ell)}\dd\ u(\ueX) \neq v(\ueX)\}$ belongs to $\Bb^{(\ell)}$,
and
\begin{equation*}
u[\ueT] = v[\ueT] \iff E_{\ueT}(D(u,v)) = 0.
\end{equation*}
\end{cor}

\begin{cor}{inverse}
Let $\ueT$ be an \FTI{} $\ell$-tuple of operators in $\HHh$ and $u \in \ffF_{\ell,1}$. If $u[\ueT]$
is one-to-one, then the range of $u[\ueT]$ is dense in $\HHh$ and there is $v \in \ffF_{\ell,1}$ for
which $v[\ueT] = (u[\ueT])^{-1}$.
\end{cor}
\begin{proof}
For simplicity, we put $S \df u[\ueT]$ and $\WwW \df \WwW''(S)$. Let $S = Q |S|$ (where $|S| =
(S^* S)^{1/2}$) be the polar decomposition of $S$ (so, $Q$ is a partial isometry whose kernel
coincides with the kernel of $S$). Since $S \in \hat{\WwW}$, we conclude that $Q \in \WwW$. So,
$Q$ is a unitary operator, being an isometry in a finite von Neumann algebra $\WwW$. Since $|S|$ is
selfadjoint and one-to-one, the range of $|S|$ is dense in $\HHh$. Thus, so is the range of $S$.
It is now an easy observation that $S^{-1}$ is affiliated with $\WwW$. We shall prove that $S^{-1}$
is the value of the functional calculus for $\ueT$.\par
Put $\zzZ \df \{\ueX \in \iiI^{(\ell)}\dd\ \det(u(\ueX)) = 0\}$. Observe that $\zzZ \in
\Bb^{(\ell)}$. We claim that $E_{\ueT}(\zzZ) = 0$. Suppose, on the contrary, that $E_{\ueT}(\zzZ)$
is nonzero. Then there in $N > 0$ such that
\begin{equation}\label{eqn:non0}
E_{\ueT}(\zzZ \cap \mmM_N^{\ell}) \neq 0.
\end{equation}
Consider a multifunction $\Psi$ on $\ddD \df \iiI^{(\ell)} \cap \mmM_N^{\ell}$ which assigns to each
$\ell$-tuple $\ueX \in \ddD$ the kernel of the matrix $u(\ueX)$. So, $\Psi(\ueX) \neq \{0\}$ \iaoi{}
$\ueX \in \zzZ$. Equipping the set of all linear subspaces of $\CCC^N$ with the Effros-Borel
structure (see \cite{ef1,ef2} or \S6 in Chapter~V in \cite{ta1} and Appendix there), we conclude
that $\Psi$ is measurable (this is a kind of folklore; it may also be simply deduced e.g.\ from
a combination of Proposition~2.4 in \cite{ern} and Corollary~A.18 in \cite{ta1}). So, it follows
from Effros' theory that there exist measurable functions $h_1,h_2,\ldots\dd \ddD \to \CCC^N$ such
that the set $\{h_n(\ueX)\dd\ n > 0\}$ is a dense subset of $\Psi(\ueX)$ for each $\ueX \in \ddD$
(to convince oneself of that, consult e.g.\ subsection~A.16 of Appendix in \cite{ta2}). We shall now
modify $h_n$ to obtain unitarily invariant functions with analogous properties. To this end, we fix
a Borel kernel $\kkK$ of $\mmM^{(\ell)}$. Let $\tau\dd \ddD \to \kkK \cap \mmM_N^{\ell}$ be
the function that assigns to each $\ueX \in \ddD$ the unique point of $\{U . \ueX\dd\ U \in \uuU_N\}
\cap \kkK$. We leave the proof that $\tau$ is Borel as an exercise. Now let $g_n \df h_n \circ
\tau$. We see that $g_n(U . \ueX) = g_n(\ueX)$ for any $\ueX \in \ddD$ and $U \in \uuU_N$. What is
more, $\{g_n(\ueX)\dd\ n > 0\} = \{h_n(\ueX)\dd\ n > 0\}$ for $\ueX \in \ddD \cap \kkK$. So, since
$\Psi(U . \ueX) = \Psi(\ueX)$ for all $\ueX \in \ddD$ and $U \in \uuU_N$, we conclude that
\begin{itemize}
\item[(d)] $\{g_n(\ueX)\dd\ n > 0\}$ is a dense subset of $\Psi(\ueX)$ for each $\ueX \in \ddD$.
\end{itemize}
But now, the sets $\{\ueX \in \ddD\dd\ g_n(\ueX) \neq 0\}$ belong to $\Bb^{(\ell)}$. Hence,
it follows from \eqref{eqn:non0} and (d) that there is $g \in \{g_n\dd\ n > 0\}$ with
\begin{equation}\label{eqn:Zg}
E_{\ueT}(\zzZ_g) \neq 0
\end{equation}
where $\zzZ_g \df \{\ueX \in \ddD\dd\ g(\ueX) \neq 0\}$. Let $w_0\dd \kkK \to \mmM$ be defined
as follows: if $\ueX \in \kkK \cap \zzZ_g$, $w_0(\ueX)$ is the matrix which corresponds
(in the canonical basis of $\CCC^N$) to a linear operator
\begin{equation*}
\CCC^N \ni \xi \mapsto \frac{\scalar{\xi}{g(\ueX)}}{\scalar{g(\ueX)}{g(\ueX)}} g(\ueX) \in \CCC^N
\end{equation*}
(where $\scalarr$ denotes the standard inner product in $\CCC^N$), and $w_0(\ueX) = 0 \in
\mmM_{d(\ueX)}$ for any $\ueX \in \kkK \setminus \zzZ_g$. Note that $u w_0$ vanishes at each point
of $\kkK$. Finally, let $w \in \ffF_{\ell,1}$ be a unique extension of $w_0$, guaranteed
by \COR{ext}. Then $u w$ is the zero function in $\ffF_{\ell,1}$. So, $u[\ueT] w[\ueT] = 0$. Since
$u[\ueT]$ is one-to-one, we see that $w[\ueT] = $. But $\zzZ_g \subset \{\ueX \in \mmM^{(\ell)}\dd\
w(\ueX) \neq 0\}$ and therefore $w[\ueT]$ is nonzero, by \eqref{eqn:Zg} and \COR{equal}. This
contradiction shows that indeed $E_{\ueT}(\zzZ) = 0$.\par
Now let $v$ be a unique function in $\ffF_{\ell,1}$ such that $v(\ueX) = (u(\ueX))^{-1}$ for $\ueX
\in \iiI^{(\ell)} \setminus \zzZ$ and $v(\ueX)$ vanishes at each point of $\zzZ$. We then have
$E_{\ueT}(D(uv,\jJ)) = E_{\ueT}(D(vu,\jJ)) = 0$. Consequently, \COR{equal} implies that both
$u[\ueT] \cdotaff v[\ueT]$ and $v[\ueT] \cdotaff u[\ueT]$ coincide with the identity operator
on $\HHh$. This means that $v[\ueT]$ is the inverse of $u[\ueT]$ as an element of the unital ring
$\ffF_{\ell,1}$. But so is $S^{-1}$ and therefore $v[\ueT] = S^{-1}$.
\end{proof}

\begin{dfn}{spectrum}
For any \FTI{} $\ell$-tuple $\ueT$, the set $\supp(E_{\ueT})$ is called the \textit{principal
spectrum} of $\ueT$.
\end{dfn}

Operator-valued spectra of arbitrary bounded Hilbert space operators were studied by Ernest
\cite{ern}, Hadwin \cite{ha1,ha2} and others (see also, e.g., \cite{kkl} and \cite{lee}).\par
Our next aim is to show that the functional calculus for \FTI{} tuples, as for normal operators, may
be obtained as an effect of integration (with respect to certain operator-valued measures). It is
worth noticing here that the way we shall do this will be ambiguous, that is, it will depend
on a kernel (of $\mmM^{(\ell)}$) we shall choose.\par
Since the $\sigma$-algebra $\Bb^{(\ell)}$ consists only of unitarily invariant sets, there is
no reasonable way to define the integral of functions in $\ffF_{\ell,1}$ with respect to spectral
measures $E_{\ueT}$ (because at the starting point, we only know how to integrate the functions
$\jJ_{\bbB}$, introduced in \DEF{meas}, which belong to $\ZZz(\ffF_{\ell,1})$, and there is
no reasonable way to approximate functions from $\ffF_{\ell,1}$ by functions from
$\ZZz(\ffF_{\ell,1})$). We shall overcome these difficulties with the aid of the next result. For
simplicity, for any Borel subset $\bbB$ of a Borel kernel of $\mmM^{(\ell)}$, we shall use
$\uuU . \bbB$ to denote the set $\{U . \ueX\dd\ \ueX \in \bbB,\ U \in \uuU_{d(\ueX)}\}$. The proof
of \LEM{ext-kernel} shows that $\uuU . \bbB \in \Bb^{(\ell)}$ for any such a set $\bbB$.\par
Everywhere below, we use $\mmM^{(*)}$ to denote the $C^*$-algebra product of all $\mmM_n$. That is,
$\mmM^{(*)}$ consists of all sequences $\xXX = (X_n)_{n=1}^{\infty}$ with $X_n \in \mmM_n$ (for any
$n$) and $(\|\xXX\| \df\,) \sup_{n\geqsl1} \|X_n\| < \infty$. Each matrix $X \in \mmM_n$ shall be
identified with the sequence in $\mmM^{(*)}$ whose $n$th entry is $X$ and all other entries are
zero. One readily checks that under such an identification $\mmM$ is a topological subspace
of $\mmM^{(*)}$ (that is, the topology of $\mmM$ coincides with that inherited from $\mmM^{(*)}$).
The unit of $\mmM^{(*)}$ shall be denoted by $\eEE$.

\begin{lem}{ext+repr}
Let $\kkK$ be a Borel kernel of $\mmM^{(\ell)}$ and $\ueT$ be an \FTI{} $\ell$-tuple of operators.
\begin{enumerate}[\upshape(A)]
\item The set function $E_{\ueT}^{\kkK}\dd \Bb(\mmM^{(\ell)}) \to \ZZz(\WwW(\ueT))$ given by
 \begin{equation*}
 E_{\ueT}^{\kkK}(\bbB) \df E_{\ueT}(\uuU . (\bbB \cap \kkK)) \qquad (\bbB \in \Bb(\mmM^{(\ell)}))
 \end{equation*}
 is a spectral measure that extends $E_{\ueT}$.
\item The function $\pi_{\kkK}\dd \mmM^{(*)} \to \WwW(\ueT)$ given by $\pi_{\kkK}(\xXX) \df
 \hat{c}_{\xXX}[\ueT]$ is a unital $*$-homomorphism where $\hat{c}_{\xXX}$, for $\xXX =
 (X_n)_{n=1}^{\infty}$, denotes a unique function in $\ffF_{\ell,1}^{bd}$ that is constantly equal
 to $X_n$ on $\kkK \cap \mmM_n^{\ell}$ for any $n$.
\end{enumerate}
\end{lem}
\begin{proof}
To prove (A), it suffices to observe that for any $\bbB_1, \bbB_2 \in \Bb(\mmM^{(\ell)})$,
the intersection of $\uuU . (\bbB_1 \cap \kkK)$ and $\uuU . (\bbB_2 \cap \kkK)$ coincides with
$\uuU . (\bbB_1 \cap \bbB_2 \cap \kkK)$ (so, in particular, disjoint sets are transformed into
disjoint sets), and that $\uuU . \bbB = \bbB$ for any $\bbB \in \Bb^{(\ell)}$; whereas
the conclusion of (B) follows from \THM{main2} and the fact that the assignment $\xXX \mapsto
\hat{c}_{\xXX}$ correctly defines a unital $*$-homomorphism of $\mmM^{(*)}$ into
$\ffF_{\ell,1}^{bd}$.
\end{proof}

Whenever $\ueT$ is an \FTI{} $\ell$-tuple on a Hilbert space $(\HHh,\scalarr)$ and $\kkK$ is a fixed
Borel kernel of $\mmM^{(\ell)}$, the function $\pi_{\kkK}$ enables us to define a multiplication
of vectors in $\HHh$ by elements of $\mmM^{(*)}$. Namely, whenever $h$ is a vector in $\HHh$ and
$\xXX$ is an arbitrary element of $\mmM^{(*)}$, we may write $\xXX \muple h$ to denote the vector
$\pi^{\kkK}(\xXX) h$. It follows from \LEM{ext+repr} that the map $\mmM^{(*)} \times \HHh \ni
(\xXX,h) \mapsto \xXX \muple h \in \HHh$ is bilinear and $\eEE \muple h = h$, $(\xXX \yYY) \muple h
= \xXX \muple (\yYY \muple h)$ and $\scalar{X \muple h}{h'} = \scalar{h}{X^* \muple h'}$ for any $h,
h' \in \HHh$ and $\xXX, \yYY \in \mmM^{(*)}$. Such a point of view allows us to look at the Hilbert
space $\HHh$ as a module over $\mmM^{(*)}$. But then, in a classical way, we may define integration
of $\mmM^{(*)}$-valued functions with respect to $\BBb(\HHh)$-valued spectral measures. The concept
of integrating vector-valued maps with respect to vector-valued measures (starting from a bilinear
map) is well developed; see, for example, a book \cite{din} or a pioneering paper \cite{bar} on this
topic as well as further \cite{go1,go2}, \cite{lew}, \cite{t-w}, \cite{s-t} and \cite{pn5} (consult
also \cite{d-u} or Chapter~IV in \cite{d-s}).\par
To define integration, we fix a Borel kernel $\kkK$ of $\mmM^{(\ell)}$ and an \FTI{} $\ell$-tuple
$\ueT$ of operators in a Hilbert space $\HHh$, and put, for simplicity, $E \df E_{\ueT}^{\kkK}$.
Also, to make the presentation more transparent, we shall use the notation ``$\xXX \muple h$''
introduced above (where $\xXX \in \mmM^{(*)}$ and $h \in \HHh$) instead of $\pi_{\kkK}(\xXX) h$.
Since $E$ takes values in the center of $\WwW(\ueT)$, we obtain
\begin{equation*}
E(\sigma) (\xXX \muple h) = \xXX \muple (E(\sigma) h) \qquad (\sigma \in \Bb(\mmM^{(\ell)}),\ \xXX
\in \mmM^{(*)},\ h \in \HHh),
\end{equation*}
which means that the operators $E(\sigma)$ are $\mmM^{(*)}$-linear (when thinking of $\HHh$
as an $\mmM^{(*)}$-module). This remark makes the concept of integration defined below more natural
and typical.\par
First assume a Borel function $v\dd \mmM^{(\ell)} \to \mmM^{(*)}$ is bounded and has a countable
range. We claim that then the series $\sum_{\xXX\in\mmM^{(*)}} \xXX \muple E(v^{-1}(\{\xXX\}))$ is
unconditionally convergent in the strong operator topology. Indeed, we may assume the range of $v$
is infinite and arrange all its elements in a one-to-one sequence $\xXX_1,\xXX_2,\ldots$ Further,
put $\sigma_n = \uuU . (v^{-1}(\{\xXX_n\}) \cap \kkK)$. For any $\xXX \in \mmM^{(*)}$, let
$\hat{c}_{\xXX}$ be as specified in part (B) of \LEM{ext+repr}. Observe that the sets $\sigma_n$ are
pairwise disjoint and
\begin{equation*}
\sum_{\xXX\in\mmM^{(*)}} \xXX \muple E(v^{-1}(\{\xXX\})) = \sum_{n=1}^{\infty} (\jJ_{\sigma_n}
\hat{c}_{\xXX_n})[\ueT].
\end{equation*}
So, since the series $\sum_{n=1}^{\infty} \jJ_{\sigma_n} \hat{c}_{\xXX_n}$ is pointwise convergent
to a function in $\ffF_{\ell,1}^{bd}$ and its partial sums are uniformly bounded (because $v$ is
bounded), we conclude that the right-hand side of the above formula converges (unconditionally)
in the strong operator topology (by (F5)). We define $\int_{\mmM^{(\ell)}}^{\pi_{\kkK}} v
\dint{E_{\ueT}^{\kkK}}$ as the limit of the aforementioned series. Note that
\begin{equation}\label{eqn:int}
\int_{\mmM^{(\ell)}}^{\pi_{\kkK}} v \dint{E_{\ueT}^{\kkK}} = \tilde{v}[\ueT]
\end{equation}
where $\tilde{v} \in \ffF_{\ell,1}^{bd}$ is given by $\tilde{v}(\ueX) \df \sum_{\xXX\in\mmM^{(*)}}
\jJ_{\uuU . (v^{-1}(\{\xXX\}) \cap \kkK)}(\ueX) \hat{c}_{\xXX}(\ueX)$. It is easy to check that
$\tilde{v}(\ueX)$ coincides with the $n$th entry of $v(\ueX)$ for any $\ueX \in \kkK \cap
\mmM_n^{\ell}$. This implies that the assignment $v \mapsto \tilde{v}$ (when $v$ runs over all Borel
functions with countable ranges) defines a continuous unital $*$-homomorphism $\Phi_0$ whose norm
does not exceed $1$. Consequently, also the function $\Psi_0$ which assigns to each such a function
$v$ the integral $\int_{\mmM^{(\ell)}}^{\pi_{\kkK}} v \dint{E_{\ueT}^{\kkK}}$ is a unital
$*$-homomorphism and $\|\Psi_0\| \leqsl 1$ (by \eqref{eqn:int}). So, both $\Phi_0$ and $\Psi_0$
extend to unital $*$-homomorphisms $\Phi$ and $\Psi$ defined on the uniform closure $\DdD$ of their
common domain $\DdD_0$. It is well-known (and follows, for example, from a classical theorem due
to Pettis \cite{pet}) that $\DdD$ consists of all functions $f\dd \mmM^{(\ell)} \to \mmM^{(*)}$ for
which $f(\mmM^{(\ell)}$ is a bounded separable subspace of $\mmM^{(*)}$ and $\psi \circ f\dd
\mmM^{(\ell)} \to \CCC$ is Borel for any continuous linear functional $\psi$ on $\mmM^{(*)}$.
(It follows from theory of Suslin spaces that the image in a metrizable space of a Polish space
under a Borel function is always separable. Thus, $\DdD$ constists precisely of all bounded Borel
functions from $\mmM^{(\ell)}$ into $\mmM^{(*)}$.) Since $\Psi_0$ is $\mmM^{(*)}$-linear (that is,
$\Psi_0(\xXX \muple v) = \xXX \muple \Psi_0(v)$ for any $v \in \DdD_0$ and $\xXX \in \mmM^{(*)}$,
which may be easily verified), we infer that also $\Psi$ is $\mmM^{(*)}$-linear. Finally, denoting
by $\tilde{v}$ and $\int_{\mmM^{(\ell)}}^{\pi_{\kkK}} v \dint{E_{\ueT}^{\kkK}}$ the values at $v \in
\DdD$ of, respectively, $\Phi$ and $\Psi$, equation \eqref{eqn:int} still hold (for all $v \in
\DdD$).\par
To reduce the number of symbols, let us simplify a suggestive notation for the integral
in \eqref{eqn:int} and write $\int^{\kkK} v \dint{E_{\ueT}}$ or $\int^{\kkK} v(\ueX)
\dint{E_{\ueT}(\ueX)}$ instead.

\begin{pro}{int}
For any Borel kernel $\kkK$ and $u \in \ffF_{\ell,1}^{bd}$,
\begin{equation}\label{eqn:int-repr}
\int^{\kkK} u \dint{E_{\ueT}} = u[\ueT].
\end{equation}
\end{pro}
\begin{proof}
We know that $\int^{\kkK} u \dint{E_{\ueT}} = \tilde{u}[\ueT]$. Thus, it suffices to check that
$\tilde{u} = u$ (for $u \in \ffF_{\ell,1}$). To this end, recall that for any $v \in \DdD_0$,
$\tilde{v}(\ueX)$ coincides with the $n$th entry of $v(\ueX)$ for all $\ueX \in \kkK \cap
\mmM_n^{\ell}$. Consequently, the same can be said about each $v \in \DdD$. So, we conclude that for
any $u \in \ffF_{\ell,1}^{bd}$, $\tilde{u}$ and $u$ coincide on $\kkK$ and hence $\tilde{u} = u$.
\end{proof}

\PRO{int} implies that $\int^{\kkK} v \dint{E_{\ueT}} = \int^{\kkK} \tilde{v} \dint{E_{\ueT}}$ for
each $v \in \DdD$. So, the value of the integral always belongs to $\WwW''(\ueT)$. One may also
readily derive from the above property \textit{bounded convergence theorem} which asserts that
$\int^{\kkK} v_n \dint{E_{\ueT}}$ converge in the $*$-strong operator topology to $\int^{\kkK} v
\dint{E_{\ueT}}$ provided $v_n \in \DdD$ are uniformly bounded and converge pointwise to $v \in
\DdD$.\par
For unbounded Borel functions $u\dd \mmM^{(\ell)} \to \mmM^{(*)}$ (they automatically have separable
ranges), the integral $\int^{\kkK} u \dint{E_{\ueT}}$ may be defined as an operator in $\hat{\WwW}$
where $\WwW = \WwW''(\ueT)$ as follows. (We argue similarly as in the proof of \PRO{affil}.) There
exists a sequence $\bbB_1,\bbB_2,\infty$ of pairwise disjoint Borel sets that cover $\mmM^{(\ell)}$
and $u$ is bounded on each of them. Let $b_n\dd \mmM^{(\ell)} \to \mmM^{(*)}$ denote
the characteristic function of $\bbB_n$. We define $\int^{\kkK} u \dint{E_{\ueT}}$
as $\sum_{n=1}^{\infty} \int^{\kkK} u b_n \dint{E_{\ueT}}$. Although such a definition is quite
natural, it is not obvious that the value of $\int^{\kkK} u \dint{E_{\ueT}}$ is independent
of the choice of the sets $\bbB_n$. We shall check it with the aid of \eqref{eqn:int}. We know from
that formula that $\int^{\kkK} u b_n \dint{E_{\ueT}} = \widetilde{u b_n}[\ueT]$ for any $n$. It may
be easily checked that $\widetilde{u b_n}$ coincides with $\jJ_{\uuU . (\bbB_n \cap \kkK)}
\tilde{u}$ (where, as before, $\tilde{u}$ is a unique function $f \in \ffF_{\ell,1}$ such that for
any $\ueX \in \kkK \cap \mmM_n^{\ell}$, $f(\ueX)$ is the $n$th entry of $u(\ueX)$). So,
$\widetilde{u b_n}[\ueT] = (\jJ_{\uuU . (\bbB_n \cap \kkK)} \tilde{u})[\ueT] \cdotaff
\jJ_{\uuU . (\bbB_n \cap \kkK)}[\ueT]$. Observing that the functions
$\jJ_{\uuU . (\bbB_n \cap \kkK)}$ are mutually disjoint and sum up to the unit matrix
of a respective degree at each irreducible $\ell$-tuple of matrices, we infer from the very
definition of $\tilde{u}[\ueT]$ that
\begin{equation*}
\int^{\kkK} u \dint{E_{\ueT}} = \sum_{n=1}^{\infty} (\jJ_{\uuU . (\bbB_n \cap \kkK)}
\tilde{u})[\ueT] \jJ_{\uuU . (\bbB_n \cap \kkK)}[\ueT] = \tilde{u}[\ueT].
\end{equation*}
In particular, the above integral is well defined, \eqref{eqn:int} holds for any Borel function
$u\dd \mmM^{(\ell)} \to \mmM^{(*)}$ and, consequently, \eqref{eqn:int-repr} holds for any $u \in
\ffF_{\ell,1}$.\par
Finally, if, for a Borel function $u\dd \mmM^{(\ell)} \to (\mmM^{(*)})^{\ell'}$, we define
$\int^{\kkK} u \dint{E_{\ueT}}$ coordinatewise, that is,
\begin{equation*}
\int^{\kkK} u \dint{E_{\ueT}} \df \Bigl(\int^{\kkK} \pi_1 \circ u \dint{E_{\ueT}},\ldots,\int^{\kkK}
\pi_{\ell'} \circ u \dint{E_{\ueT}}\Bigr)
\end{equation*}
(where $\pi_k\dd (\mmM^{(*)})^{\ell'} \to \mmM^{(*)}$ is the natural projection onto the $k$th
coordinate), then we obtain

\begin{thm}[Spectral Theorem]{spectral}
Let $\ueT$ be an \FTI{} $\ell$-tuple of operators. For any Borel kernel $\kkK$ of $\mmM^{(\ell)}$,
\begin{equation*}
\int^{\kkK} \ueX \dint{E_{\ueT}(\ueX)} = \ueT.
\end{equation*}
\end{thm}

Now we shall discuss how the multiplication ``$\muple$'' (by elements of $\mmM^{(*)}$) changes when
$\kkK$ varies. To distinguish between such multiplications induced by different Borel kernels, below
we write $\xXX \muple_{\kkK} h$ to denote $\pi_{\kkK}(\xXX) h$.

\begin{pro}{kern-vary}
Let $\kkK$ and $\kkK'$ be two Borel kernels of $\mmM^{(\ell)}$. Then there exists a function $u \in
\ffF_{\ell,1}^{bd}$ such that:
\begin{itemize}
\item $u$ is a unitary element of $\ffF_{\ell,1}^{bd}$, that is, $u^* u = u u^* = \jJ$; and
\item the function $\phi\dd \mmM^{(\ell)} \ni \ueX \mapsto u(\ueX) . \ueX \in \mmM^{(\ell)}$ is
 a well defined unitary element of $\ffF_{\ell,\ell}^{loc}$ as well as a Borel isomorphism that
 sends $\kkK$ onto $\kkK'$;
\end{itemize}
and for any \FTI{} $\ell$-tuple $\ueT$ of operators in $\HHh$:
\begin{itemize}
\item $E_{\ueT}^{\kkK'}(\sigma) = E_{\ueT}^{\kkK}(\phi^{-1}(\sigma))$ for each $\sigma \in
 \Bb(\mmM^{(\ell)})$ \textup{(}that is, $E_{\ueT}^{\kkK'}$ is the transport of $E_{\ueT}^{\kkK}$
 under $\phi$\textup{)}; and
\item $\xXX \muple_{\kkK'} h = u[\ueT] (\xXX \muple_{\kkK} (u^*[\ueT] h))$ for any $\xXX \in
 \mmM^{(*)}$ and $h \in \HHh$ \textup{(}that is, the multiplication by $\xXX$ induced by $\kkK'$ is
 unitarily equivalent to that induced by $\kkK$\textup{)}.
\end{itemize}
\end{pro}

It is worth noting here that the connection between the integrals $\int^{\kkK} u \dint{E_{\ueT}}$
and $\int^{\kkK'} u \dint{E_{\ueT}}$ for a Borel function $u\dd \mmM^{(\ell)} \to \mmM^{(*)}$ is,
in general, hard to describe (but, for $u \in \ffF_{\ell,1}$, both these integrals coincide,
which follows from \PRO{int} and the remarks preceding \THM{spectral}).

\begin{proof}
As shown in the proof of \LEM{ext}, there is a Borel function $g\dd \kkK \to \mmM$ such that
$g(\ueX) \in \uuU_{d(\ueX)}$ and $g(\ueX) . \ueX \in \kkK'$ for any $\ueX \in \kkK$. It follows from
\LEM{ext} that $g$ extends to a function $u \in \ffF_{\ell,1}^{bd}$. Then one deduces that $u(\ueX)
\in \uuU_{d(\ueX)}$ for each $\ueX \in \mmM^{(\ell)}$ and, consequently, $u^* u = u u^* = \jJ$.
Further, straightforward calculations show that the function $\phi\dd \mmM^{(\ell)} \to
\mmM^{(\ell)}$ defined by $\phi(\ueX) \df u(\ueX) . \ueX$ is compatible. It is also easily seen that
$\phi$ is locally bounded and Borel, and $\phi(\kkK) = \kkK'$. So, we infer from the compatibility
of $\phi$ that $\phi(\mmM^{(\ell)}) = \mmM^{(\ell)}$. Thus, to conclude that $\phi$ is a Borel
isomorphism, it suffices to check that it is one-to-one (again, by Suslin's theorem). To this end,
assume $\phi(\ueX) = \phi(\ueY)$. Then $d(\ueX) = d(\ueY)$ and
\begin{equation}\label{eqn:XY}
u(\ueX) . \ueX = u(\ueY) . \ueY.
\end{equation}
Taking into account the facts that $u$ is compatible and both $u(\ueX)$ and $u(\ueY)$ are unitary,
\eqref{eqn:XY} yields that $(u(\ueX) =)\, u(\ueX) . u(\ueX) = u(u(\ueX) . \ueX) = u(u(\ueY) . \ueY)
= u(\ueY) . u(\ueY)\, (= u(\ueY))$, and hence $\ueX = \ueY$ (by \eqref{eqn:XY}).\par
Now let $\ueT$ be an \FTI{} $\ell$-tuple of operators in $\HHh$. For any $\sigma \in
\Bb(\mmM^{(\ell)})$, the compatibility of $\phi$ implies that $\uuU . (\sigma \cap \kkK') = \uuU .
\phi^{-1}(\sigma \cap \kkK') = \uuU . (\phi^{-1}(\sigma) \cap \kkK)$ and hence
$E_{\ueT}^{\kkK'}(\sigma) = E_{\ueT}^{\kkK}(\phi^{-1}(\sigma))$. Finally, to establish
the connection between the multiplications (by elements of $\mmM^{(*)}$) induced by $\kkK$ and
$\kkK'$, for any $\xXX \in \mmM^{(*)}$, we shall use $\hat{c}_{\xXX}$ and $\hat{c}_{\xXX}'$
to denote the functions in $\ffF_{\ell,1}^{bd}$ that are constantly equal to the $n$th coordinate
of $\xXX$ on, respectively, $\kkK \cap \mmM_n^{\ell}$ and $\kkK' \cap \mmM_n^{\ell}$ (see part (B)
of \LEM{ext+repr}). Then, for any $\ueY \in \kkK$ and $\xXX = (X_n)_{n=1}^{\infty} \in \mmM^{(*)}$,
$\hat{c}_{\xXX}'(\phi(\ueY)) = X_{d(\ueY)} = \hat{c}_{\xXX}(\ueY)$. So, since both $\hat{c}_{\xXX}'
\circ \phi$ and $\hat{c}_{\xXX}$ belong to $\ffF_{\ell,1}$ and coincide on $\kkK$, we conclude that
$\hat{c}_{\xXX}' \circ \phi = \hat{c}_{\xXX}$. Consequently, for any $h \in \HHh$, $\xXX
\muple_{\kkK'} h = \pi_{\kkK'}(\xXX) h = (\hat{c}_{\xXX} \circ \phi)[\ueT] h =
\hat{c}_{\xXX}[\phi[\ueT]] h$. Further, observe that (by definition) $\phi = (u \pi^{(\ell)}_1
(u)^{-1},\ldots,u \pi^{(\ell)}_{\ell} (u)^{-1})$ and, consequently,
\begin{equation*}
\phi[\ueT] = (u[\ueT] \pi^{(\ell)}_1[\ueT] (u)^{-1}[\ueT],u[\ueT] \pi^{(\ell)}_{\ell}[\ueT]
(u)^{-1}[\ueT]) = u[\ueT] . \ueT
\end{equation*}
(recall that $u[\ueT]$ is a unitary operator). Further, we infer from \COR{prop} that
$\hat{c}_{\xXX}[u[\ueT] . \ueT] = u[\ueT] . \hat{c}_{\xXX}[\ueT] = u[\ueT] \hat{c}_{\xXX}[\ueT]
u^*[\ueT]$. All these remarks yield $\xXX \muple_{\kkK'} h = u[\ueT] \hat{c}_{\xXX} u^*[\ueT] h =
u[\ueT] (\xXX \muple_{\kkK} (u^*[\ueT] h))$ and we are done.
\end{proof}

The above result asserts that all operators of multiplication by a fixed element of $\mmM^{(*)}$
induced by Borel kernels of $\mmM^{(\ell)}$ are mutually unitarily equivalent. In general, there is
no reasonable way to distinguish one of them as best fitting to a fixed \FTI{} $\ell$-tuple. This
causes that many attributes (such as principal spectrum) of such tuples have to be unitarily
invariant. This is one of basic differences between general \FTI{} tuples and normal operators. For
the latters, their principal spectra coincide with usual (``algebraic''). The next result
establishes a counterpart of that property for general \FTI{} tuples, where the operator
of multiplication of vectors in $\HHh$ by $\xXX \in \mmM^{(*)}$ induced by a Borel kernel $\kkK$
shall be suggestively denoted by $\xXX \muple I_{\HHh}$ (here and below, $I_{\HHh}$ is the identity
operator on $\HHh$).

\begin{pro}{spectrum}
Let $\ueT = (T_1,\ldots,T_{\ell})$ be an \FTI{} $\ell$-tuple of operators in a separable Hilbert
space $\HHh$ and $\ueX = (X_1,\ldots,X_{\ell})$ be an irreducible $\ell$-tuple in $\mmM_p^{\ell}$.
\TFCAE
\begin{enumerate}[\upshape(i)]
\item $\ueX$ belongs to the principal spectrum of $\ueT$;
\item there exists a Borel kernel $\kkK$ of $\mmM^{(\ell)}$ such that for any $\epsi > 0$ there is
 a nonzero projection $Z \in \ZZz(\WwW''(\ueT))$ with $Z \WwW''(\ueT)$ type~$\tI_p$ and
 \begin{equation}\label{eqn:epsi}
 \|(T_k \minusaff X_k \muple I_{\HHh}) \cdotaff Z\| \leqsl \epsi \qquad (k=1,\ldots,\ell).
 \end{equation}
\end{enumerate}
\end{pro}
\begin{proof}
Assume $\kkK$ is as specified in (ii). To prove (i), we need to check that $E_{\ueT}(\bbB) \neq 0$
for any set $\bbB \in \Bb^{(\ell)}$ which contains $\ueX$ and is relatively open in $\iiI^{(\ell)}$.
To this end, fix such a set $\bbB$. Then there is a number $\delta > 0$ such that each $\ueY \in
\mmM_{d(\ueX)}^{\ell}$ with $\|\ueY - \ueX\| \leqsl \delta$ belongs to $\bbB$ (we use here
the property that $\iiI^{(\ell)}$ is open in $\mmM^{(\ell)}$). Using (ii), take a nonzero projection
$Z \in \ZZz(\WwW''(\ueT))$ for which $Z \WwW''(\ueT)$ is type~$\tI_p$ and \eqref{eqn:epsi} holds
with $\epsi = \delta$. One deduces from \PRO{sep} and \COR{equal} that there is $b \in
\ZZz(\ffF_{\ell,1}^{bd})$ with $b^* b = \jJ$ and $b[\ueT] = Z$. Moreover, since $Z \WwW''(\ueT)$ is
type~$\tI_p$, we may and do assume that $b$ vanishes at each point of $\iiI^{(\ell)} \setminus
\mmM_p^{\ell}$ (thanks to \PRO{In}).\par
Recall that $X_k \muple I_{\HHh} = \hat{c}_{X_k}[\ueT]$ and hence $(T_k \minusaff X_k \muple
I_{\HHh}) \cdotaff Z = w_k[\ueT]$ where $w_k \df (\pi^{(\ell)}_k - \hat{c}_{X_k}) b$. Further,
\PRO{sep} and \eqref{eqn:epsi} imply that $(T_k \minusaff X_k \muple I_{\HHh}) \cdotaff Z =
v_k[\ueT]$ for some $v_k \in \ffF_{\ell,1}^{bd}$ with $\|v_k\| \leqsl \delta$. So, $w_k[\ueT] =
v_k[\ueT]$; and \COR{equal} asserts that $E_{\ueT}(D(v_k,w_k)) = 0$. We conclude that
$E_{\ueT}(\zzZ_k) = I_{\HHh}$ where $\zzZ_k \df \{\ueY \in \iiI^{(\ell)}\dd\ \|w_k(\ueY)\| \leqsl
\delta\}$. Further, since $Z = b[\ueT]$ is nonzero, the value of $E_{\ueT}$ at $\zzZ_0 \df \{\ueY
\in \iiI^{(\ell)}\dd\ b(\ueY) = I_{d(\ueY)}\}$ is nonzero as well. Consequently, $E_{\ueT}(\zzZ)
\neq 0$ where $\zzZ \df \bigcap_{k=0}^{\ell} \zzZ_k$. Now let $\ueY \in \zzZ \cap \kkK$ and $k \in
\{1,\ldots,\ell\}$. Then $\|\pi^{(\ell)}_k(\ueY) - \hat{c}_{X_k}(\ueY)\| = \|w_k(\ueY)\| \leqsl
\delta$. But $\zzZ_0 \subset \mmM_p^{\ell}$ and therefore $\hat{c}_{X_k}(\ueY) = X_k$. So,
$\|\ueY - \ueX\| \leqsl \delta$, which gives $\ueY \in \bbB$ (for any $\ueY \in \zzZ \cap \kkK$).
We conclude that $\uuU . (\zzZ \cap \kkK) \subset \bbB$. Finally, since $\zzZ$ is unitarily
invariant, we see that $\uuU . (\zzZ \cap \kkK) = \zzZ$ and hence $E_{\ueT}(\bbB) \neq 0$.\par
Now assume $\ueX \in \supp(E_{\ueT})$. We shall construct a Borel kernel $\kkK$ for which (ii)
holds. To this end, we fix an arbitrary Borel kernel $\kkK'$ and define the function $\xi\dd \kkK'
\cap \mmM_p^{\ell} \to \RRR$ by $\xi(\ueY) \df \min\{\|\ueX - U . \ueY\|\dd\ U \in \uuU_p\}$. It is
easy to check that $\xi$ is continuous. Further, let $\Xi$ be a multifunction on $\kkK' \cap
\mmM_p^{\ell}$ defined by $\Xi(\ueY) \df \{U \in \uuU_p\dd\ \|\ueX - U . \ueY\| = \xi(\ueY)\}$ (for
$Y \in \kkK' \cap \mmM_p^{\ell}$). Notice that for any compact set $\llL \subset \uuU_p$, the set
$\{\ueY \in \kkK' \cap \mmM_p^{\ell}\dd\ \Xi(\ueY) \cap \llL \neq \varempty\}$ is relatively closed
in $\kkK' \cap \mmM_p^{\ell}$ and, consequently, $\{\ueY \in \kkK' \cap \mmM_p^{\ell}\dd\ \Xi(\ueY)
\cap \ggG \neq \varempty\} \in \Bb(\kkK' \cap \mmM_p^{\ell})$ for any set $\ggG \subset \uuU_p$ that
is relatively open in $\uuU_p$. So, it follows from the Kuratowski-Ryll-Nardzewski selection theorem
\cite{krn} (see also Theorem~1 in \S1 of Chapter~XIV in \cite{k-m}) that there exists a Borel
function $\tau\dd \kkK' \cap \mmM_p^{\ell} \to \uuU_p$ such that $\tau(\ueY) \in \Xi(\ueY)$ for any
$\ueY \in \kkK' \cap \mmM_p^{\ell}$. This means that
\begin{equation}\label{eqn:tau}
\|\ueX - \tau(\ueY) . \ueY\| = \min\{\|\ueX - U . \ueY\|\dd\ U \in \uuU_p\} \qquad (\ueY \in \kkK'
\cap \mmM_p^{\ell}).
\end{equation}
We define $\kkK$ as $(\kkK' \setminus \mmM_p^{\ell}) \cup \{\tau(\ueY) . \ueY\dd\ \ueY \in \kkK'
\cap \mmM_p^{\ell}\}$. It is immediate that $\kkK$ is a kernel of $\mmM^{\ell}$. The property that
$\kkK$ is Borel follows from the Suslin theorem (on the ranges of one-to-one Borel functions).
We shall now check that $\kkK$ is the kernel we searched for. To this end, fix $\epsi > 0$ and
denote by $\bbB_0$ the set of all $\ueY \in \kkK \cap \mmM_p^{\ell}$ with $\|\ueX - \ueY\| \leqsl
\epsi$. We claim that $\bbB \df \uuU . \bbB_0\ (\in \Bb^{(\ell)})$ is a neighbourhood of $\ueX$ (not
necessarily open). Indeed, if $\delta < \epsi$ is a positive real number such that the open ball
$\ssS$ around $\ueX$ of radius $\delta$ is contained in $\iiI^{(\ell)}$, then $\ssS \subset \bbB$,
which readily follows from \eqref{eqn:tau}. So, since $\ueX \in \supp(E_{\ueT})$, we infer that
$E_{\ueT}(\bbB) \neq 0$, and hence $Z \df \jJ_{\bbB}[\ueT]$ is nonzero, by \COR{equal}. We know that
$Z \in \ZZz(\WwW''(\ueT))$ is a projection such that $Z \WwW''(\ueT)$ is type $\tI_p$ (thanks
to \PRO{In}). Thus, it remains to verify \eqref{eqn:epsi}. To this end, fix $k \in \{1,\ldots,
\ell\}$, put $w \df (\pi^{(\ell)}_k - \hat{c}_{X_k}) j_{\bbB}$ and note that $(T_k \minusaff X_k
\muple I_{\HHh}) \cdotaff Z = w[\ueT]$. So, to convince oneself that \eqref{eqn:epsi} holds, it is
enough to show that $\|w\| \leqsl \epsi$ or, equivalently, that $\|w(\ueY)\| \leqsl \epsi$ for any
$\ueY \in \kkK$ (see \COR{ext}). Fix such $\ueY$. Since $j_{\bbB}(\ueY) = 0$ for $\ueY \notin \bbB$,
we may assume $\ueY \in \bbB$. Observe that $\bbB \cap \kkK = \bbB_0$ and hence $\|\ueX - \ueY\|
\leqsl \epsi$. Further, $\hat{c}_{X_k}(\ueY) = X_k$ and therefore $\|w(\ueY)\| =
\|\pi^{(\ell)}_k(\ueY) - X_k\| \leqsl \|\ueY - \ueX\| \leqsl \epsi$, and we are done.
\end{proof}

We conclude the paper with the next result, which describes a method of producing \FTI{} tuples.
For its purposes, let us agree that whenever $T_{jk}$, for $j,k=1,\ldots,N$, are operators
in a common Hilbert space $\HHh$, then $[T_{jk}]$ will denote the operator $S$ in $\HHh^N$ defined
as follows. The domain of $S$ consists of all vectors $(h_1,\ldots,h_N) \in \HHh^N$ such that $h_k$
belongs to the domain of $T_{jk}$ for any $j$ and $k$; and for $h = (h_1,\ldots,h_N)$ in the domain
of $S$, $S h$ is computed in a standard manner, as the multiplication of two matrices (of course,
one needs to write $h$ as a column). The operator $[T_{jk}]$ may not be neither closed nor densely
defined.

\begin{pro}{FTI}
If all operators $T^{(p)}_{jk}$ \textup{(}where $p=1,\ldots,\ell$ and $j,k=1,\ldots,N$\textup{)} are
affiliated with a common finite type~$\tI$ von Neumann algebra, then the operators $S_p \df
[T^{(p)}_{jk}]\ (p=1,\ldots,\ell)$ are densely defined and closable and the $\ell$-tuple of their
closures is \FTI.
\end{pro}
\begin{proof}
Let $\HHh$ and $\MmM$ denote the Hilbert space which all operators $T^{(p)}_{jk}$ act in and,
respectively, a finite type~$\tI$ von Neumann algebra which all these operators are affiliated with.
Let $\WwW$ denote the von Neumann algebra of all operators on $\HHh^N$ of the form $[S_{jk}]$ where
all $S_{jk}$ belong to $\MmM$. Since $\WwW$ is naturally $*$-isomorphic to $\MmM \bar{\otimes}
\mmM_N$, it follows (e.g.\ from Propositions~2.6.1 and 2.6.2 in \cite{sak}) that $\WwW$ is finite
and type~$\tI$. Further, the results due to Murray and von Neumann \cite{mvn} (consult also
\cite{liu} or \cite{pn4}) imply that the intersection of the domains of all $T^{(p)}_{jk}$ is dense
in $\HHh$, from which we infer that each of $S_p$ is densely defined. Now if we put
$\tilde{T}^{(p)}_{jk} \df (T^{(p)}_{kj})^*$, the above argument proves that each of the operators
$\tilde{S}_p \df [\tilde{T}^{(p)}_{jk}]$ is also densely defined. But, a straightforward calculation
shows that $\tilde{S}_p$ is contained in the adjoint of $S_p$, which implies that $S_p$ is closable.
Denote by $\bar{S}_p$ its closure. To conclude that $(\bar{S}_1,\ldots,\bar{S}_{\ell})$ is \FTI,
it is sufficient (thanks to \LEM{aff}) that $U S_p U^{-1} = S_p$ for any unitary operator $U \in
\WwW'$. But each such an operator $U$ is of the form $[V_{jk}]$ where $V_{jk} = 0$ for $k \neq j$
and $V_{jj} = V$ for any $j$ where $V$ is a unitary operator in $\MmM'$. Then, since all
$T^{(p)}_{jk}$ are affiliated with $\MmM$, we see that $V T^{(p)}_{jk} V^{-1} = T^{(p)}_{jk}$, from
which one readily deduces that $U S_p U^{-1} = S_p$, and we are done.
\end{proof}

\end{document}